\documentclass[11pt]{article}
\usepackage{amsmath}
\usepackage{amsthm}
\usepackage{amsfonts}
\usepackage{amssymb}
\usepackage{enumerate}
\usepackage{dsfont}
\usepackage{graphicx}
\usepackage{epstopdf}
\usepackage{geometry}
\usepackage{stmaryrd}
\usepackage[colorlinks=true,linkcolor=blue,urlcolor=blue,citecolor=red, hyperfigures=false]{hyperref}
\let\inf\relax \DeclareMathOperator*\inf{\vphantom{p}inf}
\let\max\relax \DeclareMathOperator*\max{\vphantom{p}max}
\let\min\relax \DeclareMathOperator*\min{\vphantom{p}min}

\numberwithin{equation}{section}
\numberwithin{figure}{section}
\newtheorem{theorem}{Theorem}[section]
\theoremstyle{plain}

\newtheorem{definition}[theorem]{Definition}

\newtheorem{lemma}[theorem]{Lemma}
\newtheorem{notation}[theorem]{Notation}

\newtheorem{proposition}[theorem]{Proposition}
\newtheorem{remark}[theorem]{Remark}

\newcommand{\be}{\begin{equation}}
\newcommand{\ee}{\end{equation}}
\def\rg{\rangle}
\def\lg{\langle}
\def\ds{\displaystyle}
\newcommand{\ep}{\varepsilon}

\newcommand{\De}{{\Delta}}
\newcommand{\RR}{{\mathbb{R}}}
\newcommand{\DD}{{\mathbb{D}}}
\newcommand{\EE}{{\mathbb{E}}}
\newcommand{\NN}{{\mathbb{N}}}
\newcommand{\PP}{{\mathbb{P}}}

\newcommand{\QQ}{{\mathbb{Q}}}

\newcommand{\CF}{{\mathcal{F}}}
\newcommand{\CG}{{\mathcal{G}}}

\newcommand{\CH}{{\mathcal{H}}}

\newcommand{\CA}{{\mathcal{A}}}

\newcommand{\CS}{{\mathcal{S}}}
\newcommand{\CL}{{\mathcal{L}}}
\newcommand{\CT}{{\mathcal{T}}}
\newcommand{\CB}{{\mathcal{B}}}
\newcommand{\CZ}{{\mathcal{Z}}}

\newcommand{\limitn}{{\underset{n\rightarrow \infty}{\longrightarrow}}}
\newcommand{\indic}{{\mathds{1}}}
\newcommand{\comment}[1]{}
\newcommand{\llb}{{\llbracket}}
\newcommand{\rrb}{{\rrbracket}}

\newcommand{\tr}{{\!\!~^\textrm{T}\!\!}}
\def\p{\vskip4truept \noindent}
\def\pp{\vskip4truept}
\title{Continuous-time limit of dynamic games with incomplete information and a more informed player.}
\author{Fabien Gensbittel\thanks{Toulouse School of Economics (GREMAQ, Universit\'{e} Toulouse 1 Capitole), 
 Manufacture des Tabacs, MF213, 21, All\'{e}e de Brienne 31015 Toulouse Cedex 6. E-mail:\href{mailto:fabien.gensbittel@tse-fr.eu}{fabien.gensbittel@tse-fr.eu}}}

\begin{document}
\maketitle

\begin{abstract}
We study a two-player, zero-sum, dynamic game with incomplete information where one of the players is more informed than his opponent. We analyze the limit value as the players play more and more frequently. The more informed player observes the realization of a Markov process $(X,Y)$ on which the payoffs depend, while the less informed player only observes $Y$ and his opponent's actions. We show the existence of a limit value as the time span between two consecutive stages goes to zero. This value is characterized through an auxiliary optimization problem and as the unique viscosity solution of a second order Hamilton-Jacobi equation with convexity constraints.
\end{abstract}
\noindent\textbf{Acknowledgments :} 
The author gratefully acknowledges the support of the Agence Nationale de la Recherche, under grant ANR JEUDY, ANR-10-BLAN 0112. The author is grateful to the editor and to an anonymous referee for carefully reading this work and making useful remarks.

%\newpage
%\setcounter{tocdepth}{1}
%\tableofcontents

\section{Introduction.}\label{sectionintroduction}
\pp
This paper contributes to the literature on zero-sum dynamic games with incomplete information, by considering the case where one player is always more informed than his opponent.
\pp
A key feature appearing in recent contributions to the field of zero-sum dynamic games is the interplay between discrete-time and continuous-time dynamic models, as in Cardaliaguet-Laraki-Sorin \cite{CLS}, Neyman \cite{neymancontinu} or Cardaliaguet-Rainer-Rosenberg-Vieille \cite{cardaetal}, where the authors consider sequences of discrete-time dynamic games in which the players play more and more frequently. 
Such an analysis is related to the study of a sequence of discretizations in time of a given continuous-time dynamic game. In the present work, we adopt this method  in order to study a continuous-time zero-sum dynamic game where one player is always more informed than his opponent and where the state variable evolves according to an exogenous Markov process.  
Precisely, we consider a model with two payoff-relevant variables $(X_t,Y_t)_{t\geq 0}$ which are evolving over time: $X$ is a Markov chain with finite state space and $Y$ is a diffusion process whose drift parameter depend on the current value of $X$. The process $X$ is privately observed by the more informed player (say player $1$) while $Y$ is publicly observed, allowing the less informed player (player $2$) to learn information about the variable $X$ during the game. We analyze the sequence of discrete-time games indexed by $n \geq 1$ with incomplete information and perfect observation of actions, where stages occur at times $\frac{q}{n}$ for $q \geq 0$. At each stage, player $1$ observes a pair of signals $(X_{\frac{q}{n}},Y_{\frac{q}{n}})$ while player $2$ only observes $Y_{\frac{q}{n}}$. The stage payoff function is assumed to depend on actions of both players and on $(X_{\frac{q}{n}},Y_{\frac{q}{n}})$. The global payoff is a discounted sum of the stage payoffs with discount factor $\lambda_n = \int_0^{1/n} re^{-rt}dt= 1- e^{-r/n}$, where $r>0$ is a given continuous-time discount rate. We assume that the stage payoffs are not observed and we study the limit value of these games as the players play more and more frequently.
\pp
We provide two characterizations for the limit value of these games as $n$ goes to infinity. The first one is a probabilistic representation formula where the optimization variable is the set of admissible belief processes for the less informed player. Such a formula already appears in Sorin \cite{sorinbook} as an illustration of the classical $Cav(u)$ theorem of Aumann and Maschler \cite{aumasch}. A similar discrete-time formula was introduced by De Meyer in \cite{demeyergeb} in order to obtain a continuous-time limit value in a class of financial games and this approach led to several extensions in continuous-time models (see Cardaliaguet-Rainer \cite{cardaexemple,carda12}, Gr\"{u}n \cite{grunstopping,grungirsanov}, Gensbittel \cite{fabiencavu,fabiencovariance}, and more recently Cardaliaguet-Rainer-Rosenberg-Vieille \cite{cardaetal} and Gensbittel-Gr\"{u}n \cite{GensbittelGrun}). This representation formula is important as it provides a characterization of optimal processes of revelation (martingales of posteriors induced by optimal strategies).
\pp
The second one is a variational characterization, the limit value is shown to be the unique viscosity solution of a second-order Hamilton-Jacobi equation with convexity constraints as introduced by Cardaliaguet \cite{cardadiff,cardadouble} and generalized in Cardaliaguet-Rainer  \cite{cardastochdiff}, Gr\"{u}n \cite{grunstopping}, Cardaliaguet-Rainer-Rosenberg-Vieille \cite{cardaetal} and Gensbittel-Gr\"{u}n \cite{GensbittelGrun}.
\section{Main results.}\label{sectionmainresults}

\begin{notation}
For any topological space $E$, $\De(E)$ denotes the set of Borel probability distributions on $E$ endowed with the weak topology and the associated Borel $\sigma$-algebra. $\delta_x$ denotes the Dirac measure on $x\in E$. Finite sets are endowed with the discrete topology and Cartesian products with the product topology. $\DD([0,\infty),E)$ denotes the set of c\`{a}dl\`{a}g trajectories taking values in $E$, endowed with the topology of convergence in Lebesgue measure. The notations $\lg,\rg$ and $|.|$ stand for the canonical scalar product and the associated norm in $\RR^m$.
\end{notation}

Let us at first describe the continuous-time game we will approximate. This description is incomplete as we do not define strategies in continuous-time. Rather, we define below strategies in the different time-discretizations of this game. The notion of value for this game will therefore be the limit value along a sequence of discretizations when the mesh of the corresponding partitions goes to zero. 
\p
We assume  that $(X_t)_{t \in [0,\infty)}$ is a continuous-time homogeneous Markov chain with finite state space $K$, infinitesimal generator $R=(R_{k,k'})_{k,k' \in K}$ and initial law $p\in \Delta(K)$. 
We identify $\Delta(K)$ with the canonical simplex of $\RR^{K}$, i.e.:
\[ \Delta(K)= \{ p\in \RR^K | \forall k\in K, p(k)\geq 0\,,\, \sum_{k\in K} p(k) =1 \}.\]
\p
Then, we define the real-valued process $(Y_t)_{t \in [0,\infty)}$ as  the unique solution of the following stochastic differential equation (SDE)
\be \label{SDE1}
\forall t\geq 0, \; Y_t= y + \int_0^t b(X_s,Y_s) ds + \int_0^t \sigma(Y_s) dW_s, 
\ee
where $(W_t)_{t \in [0,\infty)}$ is a standard Brownian motion independent of $X$ and $y \in \RR$ is a given initial condition. The process $Y$ may be seen as some noisy observation of the process $X$. 
\p
We assume that the functions $b$ and $\sigma$ in \eqref{SDE1} are bounded and Lipschitz, and that there exists $\epsilon>0$ such that for all $y \in \RR$, $\sigma(y)\geq \epsilon$. 
The state process $Z:=(X,Y)$ with values in $K\times \RR$ is a well defined Feller Markov process, with semi-group of transition probabilities denoted $(P_t)_{t\geq 0}$.
\p
Let $I,J$ denote finite action sets for the two players (players $1$ and $2$),  $g: (K \times \RR) \times  I \times J \rightarrow \RR$ a bounded payoff function which is Lipschitz with respect to the second variable, and $r>0$ a fixed discount rate.
\p
We consider the following (heuristic) zero-sum game, played on the time interval  $[0,\infty)$:
\begin{itemize}
\item Player $1$ observes the trajectory of $Z=(X,Y)$.
\item Player $2$ observes only the trajectory of $Y$.
\item They play the game $G(p,y)$ with total expected payoff for player $1$: 
\[ \EE[ \int_{0}^{+\infty} re^{-rt}g(X_t,Y_t,i_t,j_t)dt], \]
where $i_t$ (resp. $j_t$) denote the action of player $1$  at time $t$ (resp. of player $2$).
\item Actions are observed during the game (and potentially convey relevant information).
\end{itemize}
We aim at studying the value function of this game and how information is used by the more informed player when playing optimally. 
In order to achieve this goal, we introduce a sequence of time-dicretizations of the game. 
For simplicity, and without loss of generality, let us consider the uniform partition of $[0,+\infty)$ of mesh $1/n$. The corresponding discrete-time game, denoted $G_n(p,y)$ proceeds as follows: 
\begin{itemize}
\item The variable $Z_{\frac{q}{n}}=(X_{\frac{q}{n}},Y_{\frac{q}{n}})$ is observed by player $1$ before stage $q$ for $q\geq 0$.
\item The variable $Y_{\frac{q}{n}}$ is observed by player $2$ before stage $q$ for $q\geq 0$. 
\item At each stage, both players choose simultaneously a pair of actions $(i_q,j_q)\in I \times J$.
\item Chosen actions are observed after each stage. 
\item Stage payoff of player $1$ equals $g(Z_{\frac{q}{n}},i_q,j_q)$ (realized stage payoffs are not observed). 
\item The total expected payoff of player $1$ is
\[ \mathbb{E}\left[ \lambda_n \sum_{q \geq 0} (1-\lambda_n)^q g(Z_{\frac{q}{n}},i_q,j_q) \right],\]
with $\lambda_n=1-e^{-r/n}$
\end{itemize}
\begin{remark}
When $\sigma$ is constant and $b$ depends only on $X$, the observation of player 2 correspond to a normally distributed random variable with mean $\int_{\frac{q}{n}}^{\frac{q+1}{n}}b(X_{s})ds$ and variance $\frac{\sigma^2}{n}$. It may therefore be interpreted as a noisy observation of $X$.
\end{remark} 
The description of the game is common knowledge and we consider the game played in behavior strategies: at round $q$, player $1$ and player $2$ select simultaneously and independently an action $i_q\in I$ for player $1$ and $j_q\in J$ for player $2$ using some lotteries depending on their past observations. 
\p
Formally, a behavior strategy $\sigma$ for player $1$ is a sequence $(\sigma_{q})_{q\geq 0}$ of transition probabilities:
\[ \sigma_{q} :  ( (K \times \RR) \times I \times J )^{q}\times  (K \times \RR) \rightarrow \Delta(I), \]
where $\sigma_q(Z_{0},i_0,j_0,...,Z_{\frac{q-1}{n}},i_{q-1},j_{q-1},Z_{\frac{q}{n}})$ denotes the lottery used to select the action $i_q$ played at round $q$ by player $1$ when  past actions played during the game are $(i_0,j_0,...,i_{q-1},j_{q-1})$ and the sequence of observations of player $1$ is $(Z_{0},...,Z_{\frac{q}{n}})$. Let $\Sigma$ denote the set of behavior strategies for player $1$.
Similarly, a behavior strategy $\tau$ for player $2$  is a sequence $(\tau_{q})_{q\geq 0}$ of transition probabilities depending on his past observations
\[ \tau_{q} : (\RR \times I \times J )^{q} \times \RR \rightarrow \Delta(J). \]
Let $\CT$ denote the set of behavior strategies for player $2$.
\p
Let $\PP_{(n,p,y,\sigma,\tau)} \in \Delta(\DD([0,\infty),K\times \RR) \times  (I \times J)^\NN)$ denote the probability on the set of trajectories of $Z$ and actions induced by the strategies $\sigma,\tau$. 
The payoff function in $G_n(p,y)$ is defined by
\[ \gamma_n(\nu,\sigma,\tau):= \mathbb{E}_{\PP(n,p,y,\sigma,\tau)}\left[ \lambda_n \sum_{q \geq 0} (1-\lambda_n)^q g(Z_{\frac{q}{n}},i_q,j_q) \right] .\]
It is well known that the value of the game exists, i.e.
\[ V_n (p,y) := \underset{\sigma \in \Sigma}{\sup}\; \underset{\tau \in  \CT}{\inf} \; \gamma_n(p,y,\sigma,\tau)= \underset{\tau \in  \CT}{\inf} \; \underset{\sigma \in \Sigma}{\sup}\; \gamma_n(p,y,\sigma,\tau).  \]

We also need to consider  the value function $u$ of the non-revealing one-stage game $\Gamma(p,y)$, which is a finite game with payoff $g$ in which player $1$ cannot use his private information.
Precisely,
\[u(p,y) := \underset{\sigma \in \Delta(I)}{\sup} \; \underset{\tau \in \Delta(J)}{\inf} \; \sum_{i \in I}\sum_{j \in J}\sum_{k \in K} p(k)\sigma(i)\tau(j) g(k,y,i,j), \]
and the value exists (i.e. the $\sup$ and $\inf$ commute in the above formula) as it is a finite game. It follows from standard arguments that $u$ is Lipchitz in $(p,y)$.
\pp
The main results proved in sections \ref{sectioncontinuouslimit} and \ref{sectionvariational} are two different characterizations for the limit of the sequence of value functions $V_n$. 
\p
Let us now introduce some notations.
\p
\begin{notation} \ 
\begin{itemize}
\item The natural filtration $\CF^A$ of a process $(A_t)_{t\in [0,\infty)}$ is defined by $\CF^A_t= \sigma( A_s, s\leq t)$. 
The associated right-continuous filtration is denoted $\CF^{A,+}$ with $\CF^{A,+}_t:= \cap_{s >t} \CF^A_s$.
\item For any topological space $E$, $\DD( [0,\infty),E)$ denotes the set of $E$-valued c\`{a}dl\`{a}g trajectories.
\item For all $(p,y)  \in \Delta(K)\times \RR$, $\PP_{p,y} \in \Delta(\DD([0,+\infty),K\times \RR))$ denotes the law of the process $Z=(X,Y)$ with initial law $p\otimes \delta_y$.
\end{itemize}
\end{notation}
Our first main result is the following probabilistic characterization.
\begin{theorem}\label{cavu}
For all $(p,y)  \in \Delta(K)\times \RR$, 
\begin{equation}\label{eqcavu}
 V_n(p,y)  \limitn  V (p,y) := \underset{(Z_{t},\pi_t)_{t\geq 0} \in \CB(p,y)}{\max} \; \EE[ \int_0^{+\infty} re^{-rt}u(\pi_{t},Y_t)dt ],     
\end{equation}
 where $\CB(p,y)\subset \Delta(\DD([0,\infty), (K\times \RR) \times \Delta(K)))$ denotes the set of laws of c\`{a}dl\`{a}g processes $(Z_t,\pi_t)_{t\in [0,\infty)}$ such that:
\begin{itemize} 
\item $(Z_t)_{t\geq 0}$ has law $\PP_{p,y}$ and is an $\CF^{(Z,\pi)}$-Markov process.  
\item For all $t\geq 0$, for all $k \in K$, $\pi_t(k)= \PP(X_t=k |  \CF^{(\pi,Y)}_t )$.
\end{itemize}
\end{theorem}
\p
Let us comment briefly this result. We generalize here the idea that the problem the informed player is facing can be decomposed into two parts: at first he may decide how information will be used during the whole game, and then maximize his payoff under this constraint. To apply this method of decomposition, we need to identify precisely the set $\CB(p,y)$ of achievable processes of posterior beliefs on $X$ of the less informed player. The filtration $\CF^{(\pi,Y)}$ represents the information of player $2$, which observes the process $Y$ (a lower bound on information). The condition that $Z$ is $\CF^{(Z,\pi)}$-Markov reflects the fact that player $2$ cannot learn any information on the process $X$ which is not known by player $1$ (an upper bound on information) and the second condition simply says that $\pi$ represents the process of beliefs of player $2$ on $X_t$. Maximizers of the right-hand side of equation \eqref{eqcavu} represent optimal processes of revelation for the informed player and induce asymptotically optimal strategies for the informed player in the sequence of discretized games (see the proof of Theorem \ref{cavu}).
\pp
We now turn to the second characterization. Define $b(y):=(b(k,y))_{k \in K} \in \RR^K$ and for all $k \in K$ and $t\geq 0$, define the optional projection\footnote{In all the proofs, we consider only natural or right-continuous filtrations, but we adopt the same convention as in \cite{JacodShiryaev} and do not complete the filtrations to avoid complex or ambiguous notations. Note that optional projections of c\`{a}dl\`{a}g processes are well-defined and have almost surely c\`{a}dl\`{a}g paths (see appendix 1 in \cite{dellacheriemeyer}).}:
\[\chi_t(k):=\PP(X_t=k | \CF^{Y,+}_t).\]
Using Theorem 9.1 in \cite{LipsterShiryaev} (see also the Note p.360 about the Markov property), the process $\psi:=(\chi,Y)$ with values in $\RR^{K}\times \RR$ is a diffusion process satisfying the following stochastic differential equation:
\begin{equation}\label{sde}
\forall t\geq 0,\; \psi_t= \psi_0+ \int_0^t c(\psi_s)ds + \int_0^t \kappa(\psi_s) d\bar{W}_s,
\end{equation} 
where $\bar{W}$ is a standard $\CF^{Y,+}$-Brownian motion and the vectors $c(p,y)$ and $\kappa(p,y)$ in $\RR^{K+1}=\RR^{K}\times \RR$ are defined by 
\[ c(p,y):=(\tr R p, \langle p,b(y) \rangle),\]
\[\kappa(p,y):=( (\frac{p_{k}}{\sigma(y)}(b(k,y)- \langle b(y),p \rangle)  )_{k\in K} , \sigma(y)),\] 
where $\tr R$ denotes the transpose of the matrix $R$ and probabilities are seen as column vectors. We deduce from standard properties of diffusion processes that for any function $f \in C^{2}(\RR^{K+1})$ with polynomial growth (say) and for all $0\leq s\leq t$:
\begin{align*}
\EE[ f(\psi_t) | \CF^{\psi}_s ] &= f(\psi_s)+ \EE[\int_s^t A(f)(\psi_u) du | \CF^{\psi}_s ]
\end{align*}  
where $A(f)$ is the differential operator defined by (using the notation $z=(p,y)$)  
\[ Af(z) = \langle D f (z), c(z) \rangle + \frac{1}{2} \langle \kappa(z), D^2 f(z) \kappa(z) \rangle.  \]
\pp
In order to state our second main result, we need to define precisely the notion of weak solution we will use. 
Let $p \in \Delta(K)$, we define the tangent space at $p$ by 
\[ T_{\Delta(K)}(p):= \{ x \in \RR^{K} \,|\, \exists \varepsilon>0, p+\varepsilon x,p-\varepsilon x \in \Delta(K) \}.\]
Let $\CS^m$ denote the set of symmetric matrices of size $m$.
For $S \in \CS^K$ and $p\in \Delta(K)$, we define 
\[ \lambda_{\max}(p,S):= \max \left\{ \frac{\langle x, Sx \rangle}{\langle x,x\rangle} \,|\, x \in T_{\Delta(K)}(p)\setminus \{0\} \right\}\]
and by convention $\lambda_{\max}(p,X)=-\infty$ whenever $T_{\Delta(K)}(p)=\{0\}$.
\begin{theorem}\label{variational2}
$V$ is the unique continuous viscosity solution of 
\begin{equation}\label{eqpde2}
\min\{rV +H(z,DV(z),D^2V(z))  \,;\, -\lambda_{\max} (p,D_p^2V(z)) \}=0 
\end{equation}
where for all $(z,\xi,S) \in (\Delta(K)\times \RR)\times \RR^{K+1}\times \CS^{K+1}$:
\[  H(z,\xi,S):=- \langle \xi, c(z) \rangle -\frac{1}{2} \langle \kappa(z), S \kappa(z) \rangle  - ru(z),\]
and where $DV,D^2V$ denote the gradient and the Hessian matrix of $V$ and $D^2_pV(z)$ the Hessian matrix of the function $V$ with respect to the variable $p$.
\end{theorem}

Let us recall the definitions of sub and super-solutions.
\begin{definition}
We say that a bounded lower semi-continuous function $f$ is a (viscosity) supersolution of the equation \eqref{eqpde2} on $ \Delta(K) \times \RR$  if for any test function $\phi$, $C^2$ in a neighborhood of $ \Delta(K) \times \RR$ (in $\RR^K \times \RR$) such that $\phi \leq f$ on $\Delta(K) \times \RR$ with equality in $(p,y)\in\Delta(K)\times \RR$, we have
\[\lambda_{\max} (p,D_p^2\phi(p,y))\leq 0 \; \text{and} \;\, r\phi(p,y) - A(\phi)(p,y) -r u(p,y) \geq 0.\]
We say that a bounded upper semi-continuous function $f$ is a (viscosity) subsolution of the equation \eqref{eqpde2} on $ \Delta(K)\times \RR$  if for any test function $\phi$, $C^2$ in a neighborhood of $ \Delta(K) \times \RR$ (in $\RR^K \times \RR$) such that $\phi \geq f$ on $\Delta(K) \times \RR$ with equality in $(p,y)\in\Delta(K) \times \RR$, we have
\[\lambda_{\max} (p,D_p^2\phi(p,y)) < 0 \; \Rightarrow \;\, r\phi(p,y) - A(\phi)(p,y) -r u(p,y) \leq 0.\]
\end{definition}

The proof of Theorem \ref{variational2} is based on theorem \ref{cavu} and on dynamic programming. 
\p
\subsection{Possible extensions and open problems.}\label{subsectionopen}

We list below miscellaneous remarks.
\begin{itemize}
\item In comparison to \cite{cardaetal}, in the statement of Theorem \ref{cavu}, we maximize over a set of joint distributions $(Z,\pi)$ rather than on the set of induced distributions for $(\pi,Y)$, which are the only relevant variables for the computation of the objective functional. The latter set of distributions is exactly the set of joint laws of c\`{a}dl\`{a}g processes $(\pi,Y)$ such that 
for all bounded continuous function $\phi$ on $\Delta(K)\times \RR$ which are convex with respect to the first variable, we have:
\[ \forall \, 0\leq s\leq t, \; \EE[ \phi(\pi_t,Y_t) | \CF^{(\pi,Y)}_s] \geq  Q_{t-s}(\phi)(\pi_s,Y_s),\]
where $Q$ is the semi-group of the diffusion process $\psi$.
We do not prove this claim but it follows quite easily from Strassen's Theorem and the same techniques used in Lemma 4 in \cite{cardaetal} and Lemma 5.11 in \cite{GensbittelGrun}. However, such a proof would not be constructive (due to Strassen's theorem) and therefore, we do not think that this result would be more interesting stated this way. Indeed, in order to construct asymptotically  optimal strategies following the proof of Theorem \ref{cavu}, player 1 has to compute the joint law of $(Z,\pi)$ anyway (precisely the conditional law of $\pi$ given $Z$ at times $q/n$ for $q\geq0$).   
\item
One may generalize all the present results for the lower value functions to the case of infinite actions spaces $I,J$ (even if the value $u$ does not exist) by adapting the method developed in \cite{fabiencavu}. Note that the proof of the same kind of results for the upper value functions may rely on different tools as shown in \cite{fabiencavu}, and that the extension of these results in the present model remains an open question.
\item
It can be shown directly (with classical arguments) that the functions $V_n$ and $V$ are continuous. However, this does not simplify nor shorten the proofs.
\item
It is reasonable to think that Theorem \ref{cavu} can be extended to the case of a more general Feller processes $(X,Y)$, at least for diffusions with smooth coefficients.
However, such an extension leads to the following open question: is it possible to write an Hamilton-Jacobi equation in the case of a diffusion process $Z=(X,Y)$ taking values in $\RR^m\times \RR^p$? Note that such an equation would be stated in an infinite dimensional space of probability measures. 
\item
It would be interesting to try to find explicit solutions for simple examples with two states for $X$ and with simple payoff functions and simple diffusion parameters for $Y$. Such an analysis and the comparison with the examples studied in \cite{cardaetal} is left for future research. 
\end{itemize}
\section{Proof of Theorem \ref{cavu}}\label{sectioncontinuouslimit}

Recall the definition of conditional independence.
\begin{definition}
Let $(\Omega,\CA,\PP)$ a probability space and $\CF,\CG,\CH$ three sub $\sigma$-fields of $\CA$. We say that $\CF$ and $\CG$ are conditionally independent given $\CH$ if
\[ \forall F\in \CF, \forall  G \in \CG, \; \PP(F\cap G |\CH) =\PP(F|\CH)\PP(G|\CH).\]
This relation is denoted $\CF\coprod_{\CH} \CG$ and the definition extends to random variables by considering the $\sigma$-fields they generate.
\end{definition}
The next definition is related to the characterization of the Markov property in terms of conditional independence and will be useful in the sequel.
\begin{definition}\label{defna}
Given two random processes $(A_q,B_q)_{q\geq 0}$ (with values in some Polish spaces) defined on $(\Omega,\CA,\PP)$. We say that $(A_{q})_{q\geq 0}$ is non-anticipative with respect to $(B_q)_{q\geq 0}$ if
\[ \forall q\geq 0, \;\; (A_{0},...,A_{q}) \coprod_{B_{0},...,B_{q}} (B_m)_{m\geq 0}. \]
\end{definition}

The next result is a classical property of conditional independence and its proof is postponed to the appendix.
\begin{lemma}\label{representation}
Given two random processes $(A_q,B_q)_{q\geq 0}$ (with values in some Polish spaces), the process $(A_{q})_{q\geq 0}$ is non-anticipative with respect to $(B_q)_{q\geq 0}$ if and only if there exists (on a possibly enlarged probability space) a sequence of independent random variables $(\xi_q)_{q\geq 0}$ uniformly distributed on $[0,1]$ and independent of $(B_q)_{q\geq 0}$, and a sequence of measurable functions $f_q$ (defined on appropriate spaces) such that for all $q\geq 0$ 
\[ A_q = f_q (B_m,\xi_m, m\leq q). \]
\end{lemma}

The proof of Theorem \ref{cavu} is divided in two steps and relies on the technical Lemma \ref{compacite}, whose proof is postponed to the next subsection.

\p
\textbf{Step 1:} We prove that $\liminf V_n \geq V$.
\p 
Let $\sigma^*(p,y)$ and $\tau^*(p,y)$ be measurable selections of optimal strategies for player 1 and 2 respectively, in the game $\Gamma(p,y)$ with value $u(p,y)$.
\p
We start with a continuous-time process $(Z_t,\pi_t)_{t\geq 0}$ in $\CB(p,y)$. 
We consider the discrete-time process $(Z_{\frac{q}{n}},\pi_{\frac{q}{n}})_{q \geq 0}$.
Using the Markov property at times $\frac{q}{n}$, we deduce that $(\pi_{\frac{q}{n}})_{q \geq 0}$ is non-anticipative with respect to $(Z_{\frac{q}{n}})_{q \geq 0}$. We now construct a strategy $\overline{\sigma}$ in $G_n(p,y)$ depending on the process $(Z,\pi)$.
Using the conditional independence property (see Lemma \ref{representation}), there exists a sequence $(\xi_{q})_{q \geq 0}$ of independent random variables uniformly distributed on $[0,1]$ and independent from $(Z_{\frac{q}{n}})_{q \geq 0}$, and a sequence of measurable functions $(f_{q})_{q \geq 0}$ such that
\[ \pi_{\frac{q}{n}}=f_{q}((Z_{\frac{m}{n}},\xi_{m})_{m\leq q}) \; \text{for all}\; q \geq 0. \]
We define player $1$'s strategy $\sigma$ as follows: 
\[ \overline{\sigma}_q(Z_{0},...,Z_{\frac{q}{n}},\xi_0,...,\xi_q):=\sigma^*(\pi_{\frac{q}{n}}, Y_{\frac{q}{n}}). \]
This does not define formally a behavior strategy but these transition probabilities induce a joint law for $(Z_{\frac{q}{n}},i_q)_{q \geq 0}$ which can always be disintegrated in a behavior strategy (that does not depend on player $2$'s actions) since the induced process $(i_q)_{q \geq 0}$ is by construction non-anticipative with respect to $(Z_{\frac{q}{n}})_{q \geq 0}$ (using again Lemma \ref{representation}).
By taking the conditional expectation given $(Y_{\frac{\ell}{n}},\pi_{\frac{\ell}{n}},i_\ell,j_\ell)_{\ell=0,...,q}$, the payoff at stage $q$ against any strategy $\tau$ is such that:
\[ \EE_{n,p,\overline{\sigma},\tau}[ g(X_{\frac{q}{n}},Y_{\frac{q}{n}},i_q,j_q)]=\EE_{n,p,\overline{\sigma},\tau}[ \sum_{k \in K}\pi_{\frac{q}{n}}(k)g(k,Y_{\frac{q}{n}},i_q,j_q)]\geq \EE_{n,p,\overline{\sigma},\tau}[ u(\pi_{\frac{q}{n}},Y_{\frac{q}{n}})].\]
Therefore, $\overline{\sigma}$ is such that
\[ V_n(p,y)\geq \inf_{\tau}\gamma_n(p,y,\overline{\sigma},\tau)\geq \sum_{q \geq 0} \lambda_n (1-\lambda_n)^q \EE[ u(\pi_{\frac{q}{n}}, Y_{\frac{q}{n}})].\]
Define $(\tilde{\pi}^n,\tilde{Z}^n)$ as the piecewise-constant process equal to $(Z,\pi)$ at times $\frac{q}{n}$ for $q\geq 0$. Then  $(\tilde{\pi}^n,\tilde{Z}^n)$ converges in probability to $(Z,\pi)$ (see e.g. Lemma VI.6.37 in \cite{JacodShiryaev}) and therefore 
\[\sum_{q \geq 0} \lambda_n (1-\lambda_n)^q \EE[ u(\pi_{\frac{q}{n}}, Y_{\frac{q}{n}})]=\EE[\int_0^\infty re^{-rt}u(\tilde{\pi}^n_t,\tilde{Y}^n_t)dt]\limitn \EE[\int_0^\infty re^{-rt}u(\pi_t,Y_t)dt]\]
As $(Z,\pi)\in \CB(p,y)$ was chosen arbitrarily, we deduce that:
\[ \liminf_{n \rightarrow \infty} V_n(p,y) \geq V(p,y)\]
\p
\textbf{Step 2:} We prove that $\limsup V_n \leq V$.
\p
Let us fix  $(p,y)$ and let $(\varepsilon_n)_{n \geq 1}$ a positive sequence going to zero. For all $n \geq 1$, let $\sigma^n$ be an $\varepsilon_n$-optimal behavior strategy for player $1$ in $G_n(p,y)$. We will construct a strategy $\tau^n$ for player $2$ by induction such that for all $q \geq 0$ the expected payoff at round $q$ is not greater than
\begin{equation}\label{ineqsup}
\EE_{n,p,y,\sigma^n,\tau^n} [ u(\hat{p}_q,Y_{\frac{q}{n}}) + C |p_q-\hat{p}_q|_1],
\end{equation}
for some constant $C$ independent of $n$, where $|.|_1$ denotes the $\ell_1$-norm and where for all $q\geq 0$, $\hat{p}_q$ and $p_q$ denote respectively the conditional laws of $X_{\frac{q}{n}}$ given the information of player $2$ before and after playing round $q$. Precisely, for all $k \in K$:
\[ \hat{p}_q(k):=\PP_{(n,p,y,\sigma^n,\tau^n)}(X_{\frac{q}{n}}=k \mid Y_{0},i_0,j_0,...,Y_{\frac{q-1}{n}},i_{q-1},j_{q-1},Y_{\frac{q}{n}}),\]
\[ p_q(k):=  \PP_{(n,p,y,\sigma^n,\tau^n)}(X_{\frac{q}{n}}=k \mid Y_{0},i_0,j_0,...,Y_{\frac{q}{n}},i_{q},j_{q}).\]
Note that the computation of $\hat{p}_q$ does not depend on $\tau^n_q$. We can therefore define by induction $\tau^n_q:=\tau^*(\hat{p}_q,Y_{\frac{q}{n}})$. Then, inequality \eqref{ineqsup} follows directly from Lemmas V.2.5 and V.2.6 in \cite{msz}.
We now suppress the indices $(n,p,y,\sigma^n,\tau^n)$ from the probabilities and expectations.
Using that $u$ is Lipschitz with respect to $p$, we have 
 \[ \EE[ u(\hat{p}_q,Y_{\frac{q}{n}}) + C |p_q-\hat{p}_q|_1] \leq \EE[ u(p_q,Y_{\frac{q}{n}}) +2 C |p_q-\hat{p}_q|_1]\]
Define also: 
\[\tilde{p}_{q+1}:= \PP(X_{\frac{q+1}{n}}=k |Y_{0},i_0,j_0,...,Y_{\frac{q}{n}},i_{q},j_{q})= (e^{\frac{1}{n}\tr R}p_{q})(k).\]
Note that for all $q\geq 0$, the sequence $(\tilde{p}_{q+1},\hat{p}_{q+1},p_{q+1})$ is a martingale so that using Jensen's inequality.
\[ \EE[ (\tilde{p}_{q+1})^2] \leq \EE[(\hat{p}_{q+1})^2 ] .\]
On the other hand, using the previous equality, we can choose the constant $C$ so that almost surely
\[ \forall q\geq 0, \; |\tilde{p}_{q+1}-p_q|\leq \frac{C}{n}.\]
Mimicking the proof of \cite{cardaetal}, we have
\begin{align*}
\EE[\sum_{q \geq 0} \lambda_n (1-\lambda_n)^q |p_q-\hat{p}_q|_1]&=\sum_{k \in K}\sum_{q \geq 0}  \lambda_n(1-\lambda_n)^q \EE[|p_q(k)-\hat{p}_q(k)|]\\
&\leq \sum_{k \in K} \left(\sum_{q \geq 0} \lambda_n (1-\lambda_n)^q \EE[|p_q(k)-\hat{p}_q(k)|^2] \right)^{1/2}\\
&=\sum_{k \in K} \left(\sum_{q \geq 0} \lambda_n (1-\lambda_n)^q \EE[(p_q(k))^2-(\hat{p}_q(k))^2] \right)^{1/2}
\end{align*}
which is also equal to
\[\sum_{k \in K} \left(\sum_{q \geq 0} \lambda_n (1-\lambda_n)^q \EE[(p_q(k))^2-(\tilde{p}_{q+1}(k))^2+(\tilde{p}_{q+1}(k))^2-(\hat{p}_{q+1}(k))^2+(\hat{p}_{q+1}(k))^2-(\hat{p}_q(k))^2] \right)^{1/2}\]
and therefore is bounded from above by
\[\sum_{k \in K} \left(\sum_{q \geq 0} \lambda_n (1-\lambda_n)^q \EE[(\hat{p}_{q+1}(k))^2-(\hat{p}_q(k))^2] + \frac{2C}{n} \right)^{1/2} \leq K (\lambda_n+\frac{2C}{n})^{1/2}\]
We proved that:
\[ V_n(p,y) \leq \sum_{q \geq 0} \lambda_n (1-\lambda_n)^q \EE[ u(\pi_{\frac{q}{n}}, Y_{\frac{q}{n}})] +K (\lambda_n+\frac{2C}{n})^{1/2} + \varepsilon_n\]
In order to conclude the proof, we consider the continuous-time process $(\tilde{Z}^n,\tilde{\pi}^n)$ which is piecewise-constant and equal to $(Z_{\frac{q}{n}},p_q)$ at times $q/n$. Let us at first extract a subsequence of $V_n(p,y)$ which converges to $\limsup V_n(p,y)$. Then, using Lemma \ref{compacite}, there exists a further subsequence of $(\tilde{Z}^n,\tilde{\pi}^n)$ which converges in law to some process $(Z,\pi)$ in $\CB(p,y)$. 
We have therefore along this subsequence
\[\sum_{q \geq 0} \lambda_n (1-\lambda_n)^q \EE[ u(\pi_{\frac{q}{n}}, Y_{\frac{q}{n}})]=\EE[\int_0^\infty re^{-rt}u(\tilde{\pi}^n_t,\tilde{Y}^n_t)dt]\longrightarrow \EE[\int_0^\infty re^{-rt}u(\pi_t,Y_t)dt],\]
so that
\[ \limsup_{n \rightarrow \infty} V_n(p,y) \leq \EE[\int_0^\infty re^{-rt}u(\pi_t,Y_t)dt] \leq V(p,y).\]

\subsection{A technical Lemma}

In reference to the paper of Meyer and Zheng \cite{MeyerZheng}, we will denote $MZ$ the following topology on the set of c\`{a}dl\`{a}g paths.

\begin{notation}
For a separable metric space $(E,{\rm d})$, the $MZ$-topology on the set  $\DD([0,\infty), E)$ of c\`{a}dl\`{a}g functions is the topology of convergence in measure when $[0,\infty)$ is endowed with the measure $e^{-x}dx$.
The associated weak topology over the set $\Delta(\DD([0,\infty),E ))$ when $\DD([0,\infty),E )$ is endowed with the $MZ$-topology will be denoted $\CL(MZ)$. 
\end{notation}

\begin{remark}\label{product}
In contrast to the Skorokhod topology ($Sk$ hereafter), if $E=F\times F'$ is a product of separable metric spaces, the $MZ$ topology is a product topology, i.e. (as topological spaces)
\[ (\DD([0,\infty),F \times F'),MZ) =  (\DD([0,\infty),F),MZ)\times (\DD([0,\infty) , F'),MZ). \]
\end{remark}

The following remark will be used in the proofs. 
\begin{remark}\label{suslin}
If $E$ is a Polish space, the space $(\DD([0,\infty),E),MZ)$ is a separable metric space which is not topologically complete. However, its Borel $\sigma$-algebra is the same as the one generated by the $Sk$ topology and its topology is weaker than the $Sk$  topology for which the space is Polish, implying that all the probability measures are $MZ$-tight. Therefore, all the results about disintegration and measurable selection usually stated for Polish spaces and which depend only on the Borel structure apply to this space. 
\end{remark}
\p
Recall that the transition probabilities of $Z$ are denoted $(P_t)_{t\geq 0}$, i.e. for any bounded measurable function $\phi$ on $K\times \RR$, we have
\[ P_t(\phi)(z):=\EE_z[ \phi(Z_t)] = \int_{\CZ} \phi d P_t(z), \]
and that $P$ is a Feller semi-group implying that $(z,t)\rightarrow P_t(\phi)(z)$ is continuous for any bounded continuous function $\phi$.

\begin{notation}
Given a process $(Z_t)_{t\in [0,\infty)}$ of  law $\PP_{p,y}$, we define the process $(\widehat{Z}^n_t)_{t\in[0,\infty)} \in \DD([0,\infty), K\times \RR)$  by 
\[ \forall t\geq 0, \;\, \widehat{Z}^n_t := Z_{\frac{\lfloor nt\rfloor}{n}}\]
where $\lfloor a \rfloor$ denotes the greatest integer lower or equal to $a$.
\end{notation}

\begin{lemma}\label{compacite}\
Let $(p,y)$ be given, and let us consider a sequence of c\`{a}dl\`{a}g processes $(Z^n,\pi^n)$ that are piecewise constant on the partition $\{[ \frac{q}{n}, \frac{q+1}{n})\}_{q\geq 0}$ and such that 
\begin{itemize}
\item $Z^n$ has the same law as $\widehat{Z}^n$ (see the above notation).
\item $(\pi_{\frac{q}{n}})_{q\geq 0}$ is non-anticipative with respect to $(Z_{\frac{q}{n}})_{q\geq 0}$.
\item For all $t\geq 0$, for all $k\in K$, $\pi^n_{t}(k) = \PP(X^n_{t}=k |\CF^{(\pi^n,Y^n)}_t)$.
\end{itemize}
Then, the sequence $(Z^n,\pi^n)$ admits an $\CL(MZ)$-convergent subsequence and all the limit points belong to $\CB(p,y)$.
\end{lemma}

\begin{proof}
Let $\QQ_n$ denote a sequence of laws of processes $(Z^n_t,\pi^n_t)_{t\in [0,\infty)}$.
It follows from Proposition VI.6.37 in \cite{JacodShiryaev} that  $Z^n$ $\CL(MZ)$-converges to $Z$ of law $\PP_{p,y}$. 
On the other hand, Theorem 4 in \cite{MeyerZheng} together with a diagonal extraction implies that the set of possible laws for $(Z^n_t,\pi^n_t)_{t\geq 0}$ is $MZ$-relatively sequentially compact, and we may extract some convergent subsequence\footnote{Precisely, for all $T>0$ we may first apply this result to each coordinate of the processes $(Z^n_{t \wedge T},\pi^n_{t \wedge T})_{t \geq 0}$. Then, since convergent sequences are tight (see Theorem 11.5.3 in \cite{dudley} and remark \ref{suslin}), we apply Lemma \ref{tightmarginals} to  deduce that the set of laws $\{\QQ_n, n\geq 1\}$ is tight. Applying the direct part of Prohorov's theorem, which is valid for separable metric spaces, we may extract some convergent subsequence.}.
\p
Let us now prove that the limit belongs to $\CB(p,y)$.
Assume (without loss of generality) that the sequence of processes $(Z^n_t,\pi^n_t)_{t\geq 0}$ $\CL(MZ)$-converges to  $(Z_t,\pi_t)_{t\geq 0}$. Note at first that the law of $(Z_t)_{t\geq 0}$ is $\PP_{p,y}$ since the projection of the trajectories on the first coordinate is continuous (see Remark \ref{product}).
\p
Using Skorokhod's representation Theorem for separable metric spaces (see Theorem 11.7.31 in \cite{dudley}), we can assume that the processes are defined on the same probability space and that $(Z^n_t,\pi^n_t)_{t \geq 0} \overset{MZ}{\rightarrow} (Z_t,\pi_t)_{t \geq 0}$ almost surely.
Up to extracting a subsequence, we can also assume that there exists a subset $I$ of full measure in $[0,\infty)$ (i.e. $\int_I e^{-x} dx=1$) such that for all $t \in I$,  $(Z^n_t,\pi^n_t) \rightarrow (Z_t,\pi_t)$ almost surely. 
\p
We now prove that for all $t\geq 0$ and all $k\in K$
\[ \pi_t(k)=\PP( X_t=k | \CF^{(Y,\pi)}_t).\]
For any $t \in I$, any finite family $(t_1,...,t_r)$ in $I \cap [0,t]$ and any bounded continuous function $\phi$ defined on $(\RR\times \Delta(K))^r$, we have 
\[ \EE[ (\pi^n_t(k) - \indic_{X^n_t=k}) \phi(Y^n_{t_1},\pi^n_{t_1},...,Y^n_{t_r},\pi^n_{t_r})]=0.\]
It follows by bounded convergence that
\[ \EE[ (\pi_t - \indic_{X_t=k}) \phi(Y_{t_1},\pi_{t_1},...,Y_{t_r},\pi_{t_r})]=0.\]
We deduce that 
\[ \pi_t(k)=\PP( X_t=k | \CF^{(Y,\pi)}_t).\]
Given an arbitrary $t$, we take a decreasing sequence in $I$ with limit $t$ and applying Lemma \ref{backward} (see apendix), we obtain:
\[ \pi_t(k)=\PP( X_t=k | \CF^{(Y,\pi),+}_t),\]
which implies the result using the tower property of conditional expectations. 
\p
It remains to prove the Markov property. Let $t_1\leq ... \leq t_m \leq s \leq t$ in $I$, and $\phi,\phi'$ some bounded continuous functions defined on $((K\times\RR)\times \Delta(K))^m$ and $K\times \RR$, we claim that
\[ \EE [ \phi'(Z_t) \phi(Z_{t_1},\pi_{t_1},...,Z_{t_m},\pi_{t_m})]=\EE[P_{t-s}(\phi')(Z_s) \phi(Z_{t_1},\pi_{t_1},...,Z_{t_m},\pi_{t_m})].\]
Indeed, for all $n$, we have 
\[ \EE [ \phi'(Z^n_t) \phi(Z^n_{t_1},\pi^n_{t_1},...,Z^n_{t_m},\pi^n_{t_m})]=\EE[P_{\frac{\lfloor nt \rfloor-\lfloor ns \rfloor}{n}}(\phi')(Z^n_s) \phi(Z^n_{t_1},\pi^n_{t_1},...,Z^n_{t_m},\pi^n_{t_m})],\]
and the conclusion follows by bounded convergence. The property extends to arbitrary $t_1\leq ... \leq t_m \leq s \leq t$ by taking decreasing sequences in $I$ and we conclude as above that $Z$ is an $\CF^{(Z,\pi)}$ Markov process. 
\end{proof}
\p
Let us end this section with a second technical lemma whose proof is similar to Lemma \ref{compacite}.

\begin{lemma}\label{closedgraph}
The set-valued map $(p,y) \rightarrow (\CB(p,y),\CL(MZ))$ has a closed graph with compact values. 
\end{lemma}
\begin{proof}
Since $Z$ is a Feller process, the map $(p,y) \rightarrow \PP_{p,y}$ is $\CL(Sk)$-continuous (hence $\CL(MZ)$-continuous, see e.g. \cite{EthierKurtz}). We omit the rest of the proof as it follows exactly from the same arguments as Lemma \ref{compacite} with obvious modifications. 
\end{proof}

\section{The variational characterization}\label{sectionvariational}

We state at first some properties of the function $V$.
\begin{proposition}\label{propertiesV}
$V$ is upper-continuous and for all $y \in \RR$, $p\rightarrow V(p,y)$ is concave on $\Delta(K)$.
\end{proposition}

\begin{proof}
That $V$ is upper semi-continuous follows directly from Lemma \ref{closedgraph}. 
\p
Concavity follows from the same method as the well-known splitting Lemma (see e.g. Chapter V.1 in \cite{msz}). Given $y\in \RR$,  $p_1,p_2 \in \Delta(K)$ and $\lambda \in [0,1]$, $\PP_1 \in \CB(p_1,y)$ and $\PP_2 \in \CB(p_2,y)$, let us construct $\PP_{\lambda} \in \CB( \lambda p_1 + (1-\lambda)p_2,y)$ as follows. Assume that $(Z^1,\pi^1)$ and $(Z^2,\pi^2)$ are independent and of respective laws $\PP_1$ and $\PP_2$. Let $\xi$ be a random variable independent of  $(Z^1,\pi^1)$ and $(Z^2,\pi^2)$ and such that $\PP(\xi=1)=\lambda$ and $\PP(\xi=2)=1-\lambda$.
Define $(Z,\pi)$ as the process equal to $(Z^i,\pi^i)$ on $\{\xi=i\}$.  It follows easily by conditioning on $\xi$ that 
\[ \EE[ \int_0^\infty re^{-rt} u(\pi_t,Y_t)dt ]= \lambda \EE[ \int_0^\infty re^{-rt} u(\pi^1_t,Y^1_t)dt ]+(1-\lambda)\EE[ \int_0^\infty re^{-rt} u(\pi^2_t,Y^2_t)dt ]. \]
If we assume that $(Z,\pi)$ has a law $\PP_{\lambda} \in \CB(\lambda p_1 + (1-\lambda)p_2,y)$, then for any $\varepsilon>0$, we can choose $\PP_1$ and $\PP_2$ as $\varepsilon$-optimal probabilities so that
\begin{align*}
V(\lambda p_1 + (1-\lambda)p_2,y) &\geq \EE[ \int_0^\infty re^{-rt} u(\pi_t,Y_t)dt ]\\
&=\lambda \EE[ \int_0^\infty re^{-rt} u(\pi^1_t,Y^1_t)dt ]+(1-\lambda)\EE[ \int_0^\infty re^{-rt} u(\pi^2_t,Y^2_t)dt ] \\
&\geq \lambda V(p_1,y) + (1-\lambda)V(p_2,y) - \varepsilon,
\end{align*}
and this proves that $V$ is concave with respect to $p$ as $\varepsilon$ can be chosen arbitrarily small.
\p
In order to conclude, it remains therefore to prove that $(Z,\pi)$ has a law $\PP_{\lambda} \in \CB(\lambda p_1 + (1-\lambda)p_2,y)$. Note at first that $(Z_{t})_{t\geq 0}$ has law $\PP_{\lambda p_1 +(1-\lambda)p_2,y}$ by construction. 
Moreover, $\CF^{\pi,Y}_t$ is included in $\sigma(\xi)\vee \CF_{t}^{\pi^1,Y^1}\vee \CF_{t}^{\pi^2,Y^2}$. Using independence, we have therefore for all $k\in K$ and all $t\geq 0$:
\begin{align*}
 \PP( X_t=k &| \CF^{Y^1,\pi^1}_{t},\CF^{Y^2,\pi^2}_{t},\xi) \\
 &= \PP(X^1_t=k | \CF^{Y^1,\pi^1}_{t},\CF^{Y^2,\pi^2}_{t}, \xi )\indic_{\xi=1}+ \PP( X^2_t | \CF^{Y^1,\pi^1}_{t},\CF^{Y^2,\pi^2}_{t}, \xi )\indic_{\xi=2} \\
&=\PP( X^1_t=k | \CF^{Y^1,\pi^1}_{t})\indic_{\xi=1}+\PP( X^2_t=k | \CF^{Y^2,\pi^2}_{t})\indic_{\xi=2} =\pi^1_t(k)\indic_{\xi=1}+ \pi^2_t(k) \indic_{\xi=2}=\pi_t(k),
\end{align*}
and using the tower property of conditional expectations, we deduce that
\[ \pi_t(k)= \PP(X_t=k |  \CF^{(\pi,Y)}_{t}).\]
To prove the Markov property, let $s\geq t$ and $\phi$ some bounded continuous function on $\Delta(K)\times \RR$. 
As above, we have: 
\begin{align*}
\EE[\phi(Z_s) | \CF^{Z^1,\pi^1}_{t},\CF^{Z^2,\pi^2}_{t},\xi ] &=\sum_i \indic_{\xi=i}\EE[\phi(Z^i_s) | \CF^{Z^1,\pi^1}_{t},\CF^{Z^2,\pi^2}_{t},\xi ] \\
&= \sum_i \indic_{\xi=i}\EE[\phi(Z^i_s) | \CF^{Z^i}_{t}]\\
&=\sum_i \indic_{\xi=i} P_{s-t}(\phi)(Z^i_t)=P_{s-t}(\phi)(Z_t). 
\end{align*}
The conclusion follows by using the tower property of conditional expectation with the intermediate $\sigma$-field $\CF^{Z,\pi}_{t}$. 
\end{proof}

\subsection{Dynamic programming.}

\begin{notation}
In the following, we will use the notation $\EE_{p,y}$ to denote the  expectation associated to the diffusion process $\psi$ starting at time $0$ with initial position $\psi_0=(p,y)$.
\end{notation}

We now state a dynamic programming principle which will be the key element for the proof of Theorem \ref{variational2}.
\begin{proposition}\label{dynamicprog}
For all $(p,y) \in \Delta(K)\times \RR$, for all $h\geq 0$, we have
\be \label{DPPeq1}
V(p,y) = \max_{(\pi,Y) \in \CB(p,y)} \EE[ \int_0^h re^{-rt} u(\pi_t,Y_t)dt + e^{-rh}V(\pi_h,Y_h)].
\ee
As a consequence,
\be \label{subDPP}
V(p,y) \geq \EE_{p,y}[ \int_0^h re^{-rt} u(\psi_t)dt + e^{-rh}V(\psi_t)].
\ee
Moreover, if $(\pi,Y)$ is an optimal process for $V(p,y)$, then for all $h\geq 0$:
\be \label{DPPeq2}
V(p,y) = \EE[ \int_0^h re^{-rt} u(\pi_t,Y_t)dt + e^{-rh}V(\pi_h,Y_h)].
\ee
\end{proposition}

\begin{proof}
We prove at first that the maximum is reached in the right-hand side of \eqref{DPPeq1}.
Let us define the $MZ$-topology on the set $\DD([0,h], K\times \RR\times \Delta(K))$ as the convergence in Lebesgue measure of the trajectories together with the convergence of the value of the process at time $h$. Note that this topology coincides (up to an identification) with the induced topology on the subset of $\DD([0,\infty), K\times \RR\times \Delta(K))$ made by trajectories that are constant on $[h,\infty)$. Using this identification and adapting the arguments of Lemma \ref{compacite}, the set of laws of the restrictions of the processes $(Z,\pi) \in \CB(p,y)$ to the time interval $[0,h]$ is $\CL(MZ)$-sequentially relatively compact in $\Delta(\DD([0,h], K\times \RR\times \Delta(K)))$. The existence of a maximum follows since the map 
\[ \PP \in \Delta(\DD([0,h], K\times \RR\times \Delta(K))) \longrightarrow  \EE_{\PP}[ \int_0^h re^{-rt} u(\pi_t,Y_t)dt + e^{-rh}V(\pi_h,Y_h)],\]
is $\CL(MZ)$ upper-semi-continuous.
\p
We now prove \eqref{DPPeq1}. We begin with a measurable selection argument.
\p
The function $\PP \in \CB(p,y) \rightarrow J(\PP):=\EE[\int_{0}^\infty re^{-rt} u(\pi_t,Y_t)dt]$ is $\CL(MZ)$-continuous, and the set-valued map $(p,y) \rightarrow \CB(p,y) $ is $\CL(MZ)$ upper-semi-continuous. 
We deduce that the subset $O$ of the space $\Delta(K)\times \RR\times \Delta(\DD([0,\infty), (K\times \RR)\times \Delta(K)))$ defined by 
\[ O:=\{ (p,y,\PP) |  \PP \in   \CB(p,y) , J(\PP) \geq V(p,y) \} \]
is Borel-measurable (see Remark \ref{suslin}). Moreover, Lemma \ref{compacite} implies that for any $(p,y)$, there exists some $\PP$ such that $(p,y,\PP) \in O$.   
It follows therefore from Von Neumann's selection Theorem (see e.g. Proposition 7.49 in \cite{bertsekas}) that there exists an optimal universally-measurable selection $\phi$ from $\Delta(K)\times \RR$ to $\CB(p,y)$ such that for all $(p,y) \in \Delta(K)\times \RR$, $\phi(p,y) \in O$.
\p
Let $(Z,\pi) \in \CB(p,y)$ and $h\geq 0$ and let $\mu_h$ denote the joint law of $(\pi_h,Y_h)$. By construction, $\phi$ is $\mu_h$-almost surely equal to a Borel map $\tilde{\phi}$. Using Lemma \ref{gluing}, we can construct a process  $(\tilde{\pi}_s)_{s \geq h}$ (on some extension of the probability space) such that the conditional law of $(Z_{h+s},\tilde{\pi}_{h+s})_{s\geq 0}$ given $\CF^{(Y,\pi)}_{h}$ is precisely $\tilde{\phi}(\pi_{h},Y_h)$ and such that there exists a variable $U$, uniformly distributed on $[0,1]$ and independent of $(Z,\pi)$, and a measurable map $\Phi$ such that
\begin{equation}\label{construction-tilde-pi} 
(\tilde{\pi}_{s})_{s\geq h}= \Phi((Z_s)_{s \geq h},(\pi_h,Y_h),\xi).
\end{equation}
\p
Let us consider the process $(Z,\hat{\pi})$ where $\hat{\pi}$ is equal to $\pi$ on $[0,h)$ and to $\tilde{\pi}$ on $[h,\infty)$.
Using the preceding construction, if we assume  that the process $(Z,\hat{\pi})$ has a law in $\CB(p,y)$, we deduce that:
\begin{align*}
V(p,y)  &\geq \EE[\int_0^{\infty} re^{-rt} u(\hat{\pi}_t,Y_t)dt] \\
&= \EE[ \int_0^{h}re^{-rt}u(\pi_t,Y_t)dt] + e^{-rh}\EE[ \EE[ \int_{0}^\infty re^{-rt} u(\tilde{\pi}_{h+t}, Y_{h+t})dt | \CF^{(Y,\pi)}_{h}]]  \\
&\geq   \EE[ \int_0^{h}re^{-rt}u(\pi_t,Y_t)dt] + e^{-rh}\EE[ V(\pi_h,Y_h) ] 
\end{align*}
which would prove that 
\be \label{DPPineq1}
V(p,y) \geq \max_{(Z,\pi) \in \CB(p,y)} \EE[ \int_0^h re^{-rt} u(\pi_t,Y_t)dt + e^{-rh}V(\pi_h,Y_h)].
\ee 
To conclude the proof of \eqref{DPPineq1}, we now check that the process $(Z,\hat{\pi})$ has a law in $\CB(p,y)$.
\p
At first, note that $(Z_t)_{t\geq 0}$ is a Markov process with initial law $p \otimes \delta_y$ by construction. 
Let us prove that for all $t \geq 0$, $\hat{\pi}_t(k)=\PP( X_t=k | \CF^{(Y,\hat{\pi})}_{t}]$. The result is obvious by construction for $t <h$. 
For $t\geq h$, let us consider two finite families $(t_1,...,t_m)$ in $[h, t]$ and $(t'_1,...,t'_\ell)$ in $[0,h)$  and two bounded continuous function $\phi,\phi'$ defined on $(\Delta(K)\times \RR)^m$ and  $(\Delta(K)\times \RR)^\ell$. Then:
\begin{align*}
\EE[ \indic_{X_t=k}& \phi(\hat{\pi}_{t_1},Y_{t_1},...,\hat{\pi}_{t_m},Y_{t_m})\phi'(\hat{\pi}_{t'_1},Y_{t'_1},...,\hat{\pi}_{t'_\ell},Y_{t'_\ell})]
\\&= \EE[ \EE[ \indic_{X_t=k} \phi(\tilde{\pi}_{t_1},Y_{t_1},...,\tilde{\pi}_{t_m},Y_{t_m}) | \CF^{(\pi,Y)}_{h}]\phi'(\pi_{t'_1},Y_{t'_1},...,\pi_{t'_\ell},Y_{t'_\ell})] \\
&= \EE[\tilde{\pi}_t(k) \phi(\tilde{\pi}_{t_1},Y_{t_1},...,\tilde{\pi}_{t_m},Y_{t_m})\phi'(\pi_{t'_1},Y_{t'_1},...,\pi_{t'_\ell},Y_{t'_\ell})]\\
&= \EE[\hat{\pi}_t(k)  \phi(\hat{\pi}_{t_1},Y_{t_1},...,\hat{\pi}_{t_m},Y_{t_m})\phi'(\hat{\pi}_{t'_1},Y_{t'_1},...,\hat{\pi}_{t'_\ell},Y_{t'_\ell})].
\end{align*}
This property extends to bounded measurable functions of any finite family $(t_i)$ in $[0,t]$ by monotone class and we deduce that $\hat{\pi}_t(k)=\PP( X_t=k | \CF^{(Y,\hat{\pi})}_{t})$.
\p
We now prove the Markov property.  For $t \geq 0$, we have to prove that 
\be \label{cond-ind-1}(Z_s)_{s \geq t} \coprod_{(Z_s)_{s \in [0,t]}} (\hat{\pi}_s)_{s \in [0,t]}.\ee
The case $t< h$ follows directly by construction. Let us consider the case $t \geq h$. 
%All the following conditional independence properties follow directly from the characterization of conditional independence in terms of conditional laws recalled in Lemma \ref{condind}, so that their detailed proof is omitted. 
\p
At first, since the conditional law of $(Z_{s},\tilde{\pi}_{s})_{s\geq h}$ given $(\pi_h,Y_h)$ belongs to $\CB(\pi_h,Y_h)$, we have:
\be \label{cond-ind-2} (\tilde{\pi}_s)_{s \in[h,t]} \coprod_{(Z_s)_{s \in[h,t]},(\pi_h,Y_h)} (Z_s)_{s \geq t}.\ee
Using \eqref{construction-tilde-pi} and that $Z$ is an $\CF^{Z,\pi}$-Markov process, we also have
\be \label{cond-ind-3} 
(\tilde{\pi}_s)_{s \in[h,t]} \coprod_{(Z_s)_{s \in[h,t]},(\pi_h,Y_h)} (Z_s,\pi_s)_{s \in[0,h]},
\ee
\be \label{cond-ind-4} 
 (\tilde{\pi}_s)_{s \in[h,t]} \coprod_{(Z_s)_{s \geq t},(Z_s)_{s \in[h,t]},(\pi_h,Y_h)} (Z_s,\pi_s)_{s \in[0,h]}.
\ee
From the characterization of conditional independence in terms of conditional laws recalled in Lemma \ref{condind}, properties \eqref{cond-ind-2}, \eqref{cond-ind-3} and \eqref{cond-ind-4} together imply that:
\be \label{cond-ind-5}
(\tilde{\pi}_s)_{s \in[h,t]} \coprod_{(Z_s,\pi_s)_{s \in[0,h]},(Z_s)_{s \in[h,t]},(\pi_h,Y_h)} (Z_s)_{s \geq t}
\ee
Using again the fact that $Z$ is an $\CF^{Z,\pi}$-Markov process, we also have
\be \label{cond-ind-6} 
(Z_s)_{s \geq t} \coprod_{(Z_s)_{s \in[0,t]}} (\pi_s)_{s \in[0,h]}
\ee
Finally \eqref{cond-ind-5} and \eqref{cond-ind-6} imply 
\be \label{cond-ind-7} (Z_s)_{s \geq t} \coprod_{(Z_s)_{s \in[0,t]}} \left( (\pi_s)_{s \in[0,h]}, (\tilde{\pi}_s)_{s \in [h,t]}\right),
\ee
from which we deduce \eqref{cond-ind-1} since $(\hat{\pi}_s)_{s \in [0,t]}$ is a function of $\left( (\pi_s)_{s \in[0,h]}, (\tilde{\pi}_s)_{s \in [h,t]}\right)$.
%Given two finite families $(t_1,...,t_m)$ in $[h, s]$ and $(t'_1,...,t'_\ell)$ in $[0,h)$ and three bounded continuous functions $g,\phi,\phi'$ defined on $\Delta(K)\times \RR$, $(\Delta(K)\times \RR)^m$ and  $(\Delta(K)\times \RR)^\ell$ respectively, we deduce as above that
%\begin{align*}
% \EE[ g(Z_t)\phi(\hat{\pi}_{t_1},Y_{t_1},...,\hat{\pi}_{t_m},Y_{t_m})&\phi'(\hat{\pi}_{t'_1},Y_{t'_1},...,\hat{\pi}_{t'_\ell},Y_{t'_\ell})] \\
% &= \EE[ \EE[g(Z_t) \phi(\tilde{\pi}_{t_1},Y_{t_1},...,\tilde{\pi}_{t_m},Y_{t_m}) | \CF^{(\pi,Y)}_{h}]\phi'(\pi_{t'_1},Y_{t'_1},...,\pi_{t'_\ell},Y_{t'_\ell})] \\
%&= \EE[P_{t-s}g(Z_s) \phi(\tilde{\pi}_{t_1},Y_{t_1},...,\tilde{\pi}_{t_m},Y_{t_m})\phi'(\pi_{t'_1},Y_{t'_1},...,\pi_{t'_\ell},Y_{t'_\ell})]\\
%&= \EE[P_{t-s}g(Z_s) \phi(\hat{\pi}_{t_1},Y_{t_1},...,\hat{\pi}_{t_m},Y_{t_m})\phi'(\hat{\pi}_{t'_1},Y_{t'_1},...,\hat{\pi}_{t'_\ell},Y_{t'_\ell})].
%\end{align*}
%This property extends to bounded measurable functions of any finite family $(t_i)$ in $[0,s]$ by monotone class and we deduce that $P_{t-s}g(Z_s)=\EE[g(Z_t) | \CF^{(Y,\hat{\pi})}_{t}]$.
This concludes the proof of \eqref{DPPineq1}.
\p
In order to conclude the proof of \eqref{DPPeq1}, we now prove the reverse inequality.
\p
Let $(Z,\pi)$ be an admissible process and $h>0$. We check easily that  the conditional law of $(Z_{h+s},\pi_{h+s})_{s \geq 0}$ given $\CF^{(\pi,Y)}_h$ belongs almost surely to $\CB(\pi_h,Y_h)$. It follows that 
\begin{align*}
\EE[\int_0^{\infty} re^{-rt} u(\pi_t,Y_t)dt] &= \EE[\int_0^{h} re^{-rt} u(\pi_t,Y_t)dt]+ e^{-rh}\EE[ \EE[\int_0^{\infty} re^{-rt} u(\pi_{h+t},Y_{h+t})dt | \CF^{(\pi,Y)}_h ]]\\
&\leq \EE[\int_0^{h} re^{-rt} u(\pi_t,Y_t)dt]+ e^{-rh}\EE[V(\pi_h,Y_h)].
\end{align*}
The conclusion follows by taking the supremum over all admissible processes $(Z,\pi)$.
\p
The inequality \eqref{subDPP} follows directly from \eqref{DPPeq1}.  Precisely, given a process $Z$ with initial law $p\otimes \delta_y$, define $\pi$ by $\pi_t(k):=\chi_t(k)= \PP[ X_t=k | \CF^{Y,+}_{t}]$ (optional projection). As explained before, $(Z,\pi)$ has a law in $\CB(p,y)$ and $(\pi,Y)$  is a diffusion process of semi-group $Q$. 
\p 
We finally prove \eqref{DPPeq2}. If $(Z,\pi) \in \CB(p,y)$ is an optimal process (existence of a maximum follows from Lemma \ref{closedgraph}), then using the same arguments as above, we have for all $h\geq 0$:
\[V(p,y)=\EE[\int_0^{\infty} re^{-rt} u(\pi_t,Y_t)dt]\leq \EE[\int_0^{h} re^{-rt} u(\pi_t,Y_t)dt]+ e^{-rh}\EE[V(\pi_h,Y_h)],\]
and the conclusion follows from \eqref{DPPeq1}.
\end{proof}

\subsection{Proof of Theorem \ref{variational2}.}\label{sectionfiltering}

\begin{proof}[Proof of theorem \ref{variational2}]
The proof is divided in two parts showing respectively that the lower semicontinuous envelope $V_*$ of $V$ is subsolution and that $V$ is supersolution  of \eqref{eqpde2}. Uniqueness and continuity will follow from the comparison result (Theorem  \ref{comparisonthm}) whose proof is postponed to the appendix.
\p
\textbf{part 1:} We prove that the lower semicontinuous envelope of $V$, denoted $V_*$, is a supersolution of \eqref{eqpde2}.
\p
Let $\phi$ be any smooth test function such that $\phi \leq V_*$ with equality in $(p,y) \in  \Delta(K)\times \RR$. As $V_*$ is bounded, we may assume without loss of generality  that $\phi$ is  bounded.
Consider a sequence $(p_n,y_n) \rightarrow (p,y)$ such that $V(p_n,y_n) \rightarrow V_*(p,y)$. From \eqref{subDPP}, we deduce that 
\[V(p_n,y_n)- e^{-rh}\EE_{p_n,y_n} [\phi(\psi_h)]- \EE_{p_n,y_n}[\int_{0}^{h}re^{-rs} u( \psi_s)ds]  \geq 0.\] 
Letting $n \rightarrow \infty$, we obtain that (recall that $\psi$ is a Feller process):
\[ \phi(p,y)-e^{-rh}\EE_{p,y} [\phi(\psi_h)]- \EE_{p,y}[\int_{0}^{h} re^{-rs}u( \psi_s)ds]  \geq 0.\] 
Applying It\^{o}'s formula, we have
\[ \EE_{p,y}[\phi(\psi_h)]= \phi(p,y) + \EE_{p,y}[ \int_0^h A(\phi)(\psi_s)ds].\]
Dividing by $h$ and letting then $h \rightarrow 0$, it follows from usual arguments that 
\be \label{supersol1} 
r\phi(p,y) - A(\phi)(p,y) -r u(p,y) \geq 0.
\ee
Let us prove that $V_*$ is concave with respect to $p$.
Let $y\in \RR$ and  $p=\lambda p_1+(1-\lambda)p_2$ for some $p,p_1,p_2 \in \Delta(K)$ and $\lambda\in[0,1]$. Let $(p^n,y^n)$ a sequence converging to $(p,y)$ such that $V(p^n,y^n) \rightarrow V_* (p,y)$. Then, there exists $p_1^n,p_2^n \in \Delta(K)$  such that $p^n=\lambda p_1^n + (1-\lambda)p_2^n$ and $(p_1^n,p_2^n)\rightarrow (p_1,p_2)$ (it is for example a consequence of Lemma 8.2 in \cite{LarakiSplitting}). It follows that
\[ V(p^n,y^n) \geq \lambda V(p_1^n,y^n) +(1-\lambda)V(p_2^n,y^n).\]
By letting $n\rightarrow \infty$ and using the definition of $V_*$, we deduce that
\[ V_*(p,y) \geq \lambda V_*(p_1,y) +(1-\lambda)V_*(p_2,y),\]
which proves that $V_*$ is concave. We deduce that $\lambda_{\max}(p,D^2\phi_p(p,y)) \leq 0$, and together with \eqref{supersol1} this concludes the proof of the supersolution property.
\p
\textbf{part 2:} We prove that $V$ is subsolution of \eqref{eqpde2}.
\p
Let $\phi$ be smooth test function such that $\phi \geq V$ with equality at $\bar{z}=(\bar{p},\bar{y})$. We have to prove that if $\lambda_{\max}(\bar{p},D^2_p\phi(\bar{z}))<0$, then $rV(\bar{z}) -A(\phi) (\bar{z}) - ru(\bar{z}) \leq 0$.
\p
Using Proposition \ref{dynamicprog}, let $(Z,\pi) \in \CB(\bar{z})$ be an optimal process, so that for all $h\geq 0$, we have 
\be \label{DP1} 
V(\bar{z}) = \EE[ \int_0^h r e^{-rs} u(\pi_s,Y_s)ds + e^{-rh} V(\pi_h,Y_h)]. 
\ee
Since $\lambda_{\max}(\bar{p},D^2_p\phi(\bar{z}))<0$ (see e.g. the proof of Theorem 3.3. in \cite{cardadouble}), there exists $\delta>0$ such that for all $z=(p,\bar{y})$ with $p \in \Delta(K)$ such that $p-\bar{p} \in T_{\Delta(K)}(\bar{p})$, we have:
\[ V(z) \leq V(\bar{z}) + \langle D_p\phi(\bar{z}),p-\bar{p}\rangle - \delta |p- \bar{p}|^2.\]
As $\EE[\pi_0]=\bar{p}$, the variable $\pi_0$ belongs almost surely to the smallest face of $\Delta(K)$ containing $\bar{p}$ so that $\pi_0 - \bar{p} \in T_{\Delta(K)}(\bar{p})$. On the other hand, $Y_0=\bar{y}$ so that \eqref{DP1} with $h=0$ implies
\[ V(\bar{z}) = \EE[ V(\pi_0,\bar{y})] \leq V(\bar{z})  - \delta \EE[|\pi_0- \bar{p}|^2].\]
We deduce that $\pi_0=\bar{p}$ almost surely. 
\p
Recall the definition of the process $\chi$ as an optional projection: 
\[ \forall k \in K, \forall s\geq 0, \; \chi_s(k)=\PP(X_s=k | \CF^{Y,+}_s).\]
Lemma \ref{backward} implies that  $\pi_s(k)= \PP(X_s=k | \CF^{(\pi,Y),+}_s)$, and we deduce that $\EE[\pi_s | \CF^{Y,+}_s]=\chi_s$ using the tower property of conditional expectations. Coming back to \eqref{DP1}, Jensen's inequality implies: 
\[ V(\bar{z})=\EE[ \int_0^h r e^{-rs} u(\pi_s,Y_s)ds + e^{-rh} V(\pi_h,Y_h)] \leq \EE[ \int_0^h r e^{-rs} u(\pi_s,Y_s)ds + e^{-rh} V(\chi_h,Y_h)].\]
Since $V \leq \phi$, we obtain
\[ V(\bar{z})=\phi(\bar{z}) \leq \EE[ \int_0^h r e^{-rt} u(\pi_s,Y_s)ds + e^{-rh} \phi(\chi_h,Y_h)].\]
Dividing the above inequality by $h$, and letting $h$ go to zero, it follows from the usual arguments (using that $\pi_s \rightarrow \pi_0$ when $s\rightarrow 0$, and It\^{o}'s formula)  that:
\[ rV(\bar{z}) -A(\phi)(\bar{z}) - ru(\bar{z}) \leq 0. \]
\end{proof}

\appendix
\section{Technical Proofs and auxiliary tools.}

\subsection{Proofs of Lemma \ref{representation}} 
Let us now recall some properties of conditional independence.
As we will manipulate conditional laws, we introduce a specific notation in order to shorten statements and proofs. 
\begin{notation}
Let $E$ be a Polish space and $A$ be an  $E$-valued random variable defined on some probability space  $(\Omega,\CA,\PP)$.  
\begin{itemize}
\item $\llb A \rrb$ denotes the law of $A$. 
\item Given a $\sigma$-field $\CF \subset \CA$, $\llb A \mid \CF \rrb$ denotes a version of the conditional law of $A$ given $\CF$, hence an $\CF$-measurable random variable  with values in $\Delta(E)$ (see e.g. \cite{bertsekas} Proposition 7.26 for this last point). 
\end{itemize}
\end{notation}

%\begin{lemma}\label{propci}
%Let $E$ be a Polish space and $A,B$ be $E$-valued random variables defined on  $(\Omega,\CA,\PP)$.  Let $\CF \subset \CG$ be two sub $\sigma$-fields of $\CA$.
%\begin{itemize}
%\item If $B$ is independent of $(A,\CF)$, then $\llb A | \CF \rrb = \llb A | \CF,B \rrb$. 
%\item If $\llb A | \CF \rrb = \llb A \mid \CG \rrb$, then for any $\sigma$-field $\CH$ such that $\CF \subset \CH %\subset \CG$, we have 
%\[ \llb A | \CF \rrb = \llb A \mid \CH\rrb = \llb A \mid \CG \rrb\]
%\end{itemize}
%\end{lemma}

\begin{lemma}\label{condind}{~}\
\begin{itemize}
\item Let $A,B,C$ be three random variables (with values in some Polish space)  defined on the same probability space. $A$ is independent of $B$ conditionally on $C$ if and only if $\llb B | C \rrb = \llb B | C,A \rrb$. 
\item $A \coprod_C B$ if and only if   there exists (on a possibly enlarged probability space) a random variable $\xi$ uniform on $[0,1]$ independent of $(A,C)$, and  a measurable function $f$ such that $B=f(C,\xi)$. 
\end{itemize}
\end{lemma}

\begin{proof}
See  Proposition 5.6 and 5.13 in \cite{kallenberg}. 
\end{proof}

\begin{proof}[Proof of Lemma \ref{representation}]
The ``if'' part is obvious. Let us prove the ``only if'' part. 
For $q=0$, this is just Lemma \ref{condind}. However, we need to be more precise on how to construct this variable. We assume that there exists a family of independent variables $(\zeta_0,...,\zeta_n)$ uniformly distributed on $[0,1]$ and independent of $(A_0,B_0,...,A_n,B_n)$. Then, the variable $\xi_0$ given by Lemma \ref{condind} can be constructed as a function of $(A_0,B_0,\zeta_0)$ (see the proof of Proposition 5.13 in \cite{kallenberg}). Let us now proceed by induction and assume the above property is true for $p\leq q$ and that $\xi_p$ is measurable with respect to $(A_0,B_0,\zeta_0,...,A_p,B_p,\zeta_p)$. 
Since 
\[(A_0,...,A_{q+1})\coprod_{(B_0,...,B_{q+1})} (B_0,...,B_{n}),\]
 we have 
\[ \llb B_0,...,B_{n} | B_0,...,B_{q+1},A_0,...,A_{q+1} \rrb = \llb B_0,...,B_{n} | B_0,...,B_{q+1} \rrb. \]
We deduce that
\[  \llb B_0,...,B_{n} | B_0,...,B_{q+1},A_{q+1} \rrb=\llb B_0,...,B_{n} | B_0,..,B_{q+1} \rrb.  \]
Using now the induction hypothesis and independence, we also have
\begin{align*} 
&\llb B_0,...,B_{n} | B_0,...,B_{q+1},\xi_0,...,\xi_{q} \rrb =\llb B_0,...,B_{n} | B_0,...,B_{q+1} \rrb,  \\
& \llb B_0,...,B_{n} | B_0,...,B_{q+1},\xi_0,...,\xi_q,A_{q+1} \rrb =  \llb B_0,...,B_{n} | B_0,...,B_{q+1},A_{q+1} \rrb.  
\end{align*}
Finally, we deduce that $A_{q+1} \coprod_{(\xi_0,...,\xi_q,B_0,...,B_{q+1})} (B_0,...,B_{n})$ and the result follows then by applying Lemma \ref{condind}. 
\end{proof}

\subsection{Auxiliary Tools}\label{auxiliary}

The following lemma is classical.
\begin{lemma} \label{tightmarginals}
Let $E$,$E'$ be two separable metric spaces and $A$,$A'$ two tight (resp. closed, convex) subsets of  $\Delta(E)$ and $\Delta(E')$.
Then the set $\mathcal{P}(A,A')$ of probabilities on $E\times E'$ having marginals in the sets $A$ and $A'$ is itself tight (resp. closed, convex).
\end{lemma}
\begin{proof}
Let us prove the tightness property. Let $\mu \in A$, $\nu \in A'$ and $\pi \in \mathcal{P}(\mu,\nu)$. By assumption, for any $\varepsilon>0$ there is a compact $K_{\varepsilon}$ of $E$, independent of the choice of $\mu$ in $A$, such that $\mu(E/K_{\varepsilon})\leq \varepsilon$, and a compact $K'_{\varepsilon}$, independent of the choice of  $\nu$ in $A'$ such that $\nu(E'/K'_{\varepsilon})\leq \varepsilon$. Then for any pair of random variables $(U,V)$ of law $\pi$:
\[ \mathbb{P}[ (U,V)\notin K_{\varepsilon}\times L_{\varepsilon} ] \leq \mathbb{P}[U\notin K_{\varepsilon}]+ \mathbb{P}[ V\notin L_{\varepsilon} ] \leq 2\varepsilon \]
The closed and convex properties follow directly from the continuity and linearity of the application mapping $\pi$ to its marginals.
\end{proof}

The following theorem is well-known and allows to construct variables with prescribed conditional laws.
\begin{theorem}\label{dubinsapp}(Blackwell-Dubins \cite{blackwelldubins}) \\
Let $E$ be a polish space with $\Delta(E)$ the set of Borel probabilities on $E$,and $([0,1],\CB([0,1]),\lambda)$ the unit interval equipped with Lebesgue's measure. There exists a measurable mapping
\[ \Phi : \Delta(E)\times [0,1]  \longrightarrow E  \]
such that for all $\mu \in \Delta(E)$, the law of $\Phi(\mu,U)$ is $\mu$ where $U$ is the canonical element in $[0,1]$.
\end{theorem}
In the proof of Proposition \ref{dynamicprog}, we use indirectly this result together with  a disintegration theorem. Precisely:
\begin{lemma}\label{gluing}
Let $E,F$ be Polish spaces,  $(\Omega,\CA,\PP)$ be some probability space, $Y$ be an $E$-valued random variable defined on $\Omega$, and $\CF$ a sub-$\sigma$-field of $\CA$. Assume that $f$ is an $\CF$-measurable map from $\Omega$ to $\Delta(E\times F)$ such that the marginal $f_1(\omega) \in \Delta(E)$ of $f(\omega)$ on the first coordinate is a version of the conditional law of $Y$ given $\CF$. Then, (up to enlarging the probability space, there exists a random variable $Z$ such that $f(\omega)$ is a version of the conditional law of $(Y,Z)$ given $\CF$.
\end{lemma}
\begin{proof}
Up to enlarging the probability space, we may assume that there exists some random variable $U$ uniformly distributed on $[0,1]$ and independent of $(Y,\CF)$. One can define using Theorem \ref{dubinsapp} a variable $(\tilde{Y},\tilde{Z})=\Phi(f(\omega),U)$ having the property that $f_1(\omega)$ is a version of the conditional law of $\tilde{Y}$ given $\CF$. Let $g(\omega, \tilde{Y})$ be a version of the conditional law of $\tilde{Z}$ given $(\CF,\tilde{Y})$, it follows easily that $Z=\Phi(g(\omega,Y),U)$ fulfills the required properties.
\end{proof}

The next Lemma is a generalized martingale backward convergence theorem directly adapted from the corresponding result for classical forward martingales that can be found in chapter III of \cite{msz}.
\begin{lemma}\label{backward}
Let $(X_n)_{n\geq 0}$ be an uniformly bounded sequence of real-valued random variables defined on some probability space $(\Omega,\CA,\PP)$. Let $(\CF_n)_{n \geq 0}$ be a decreasing sequence of sub $\sigma$-fields of $\CA$. Assume that $(X_n)_{n\geq 0}$ converges almost surely to some variable $X$, then $(\EE[ X_n | \CF_n ])_{n \geq 0}$ converges almost surely to $\EE[X | \bigcap_{n \geq 0} \CF_n ]$.
\end{lemma}
\begin{proof}
Define $X^+_n=\sup_{m\geq n} X_m$ and $Y^+_n=\EE[X^+_n |\CF_n]$. The sequence $X^+_n$ is non-increasing with limit $X$ and we have
\[ Y^+_{n+1}= \EE[ X^+_{n+1} | \CF_{n+1} ] \leq \EE[X^+_n | \CF_{n+1}]= \EE[ Y_n^+ |\CF_{n+1}].\]
$Y^+_n$ is therefore a backward sub-martingale and converges almost surely to some variable $Y^+$ (see e.g. Theorem 30 p.24 in \cite{dellacheriemeyer}) which is $\bigcap_{n \geq 0} \CF_n$-measurable. Similarly, define $X_n^-=\inf_{m\geq n}X_m$, and $Y^-_n=\EE[X^-_n |\CF_n]$. Then $Y^-_n$ is a backward supermartingale which converges almost surely to $Y^-$. To conclude, note that 
\[ Y_n^- \leq \EE[ X | \CF_n ] \leq Y_n^+,\]
and that  $\EE[ Y_n^+-Y_n^- ]=\EE[X_n^+-X_n^-] \rightarrow 0$
by bounded convergence. Since $\EE[ X | \CF_n ]$ converges almost surely to $\EE[X|\bigcap_{n \geq 0} \CF_n ]$, we deduce that $\EE[ X_n | \CF_n ]$ converges almost surely to $\EE[X|\bigcap_{n \geq 0} \CF_n ]$ as $Y_n^- \leq \EE[ X_n | \CF_n ] \leq Y_n^+$.
\end{proof}

\subsection{comparison}
In this section we adapt the comparison principle given in \cite{cardaetal} for super solutions and sub solutions of \eqref{eqpde2}. 

\begin{remark} \label{extensionLip}
Note that  the process $\chi$ takes values in $\Delta(K)$, and that our assumptions on $b$ and $\sigma$ imply that the functions $c$ and $\kappa$ are Lipschitz continuous and bounded on $ \Delta(K)\times \RR$. In the following, we will assume without loss of generality that the functions $c$ and $\kappa$ are bounded  and Lipschitz on the whole space $\RR^{K+1}$ (the explicit formula cannot  be used directly since the resulting functions would be unbounded and only locally Lipschitz). Similarly, we assume that the function $u$ is bounded and Lipschitz on the whole space $\RR^{K+1}$. 
\end{remark}

With our assumptions on $c$ and $\kappa$, it is well known (see e.g. \cite{userguide}, p.19) that there exists a constant $C$ (depending on the Lipschitz constants of $c,\kappa,u$) such that for any $\eta >0$, $z,z' \in \Delta(K) \times \RR$, $\xi \in \RR^{K+1}$ and symmetric matrices $S,S' \in \CS^{K+1}$ with
\[ \begin{pmatrix}S & 0 \\ 0 & S' \end{pmatrix}\leq  \eta \begin{pmatrix}I & -I \\ -I & I \end{pmatrix}, \]
we have
\[ |u(z)- u(z')| \leq C |z-z'|\]
\[ |\lg b(z), \xi \rg - \lg b(z'), \xi \rg | \leq C |\xi| |z-z'|,\]
\[ -\frac{1}{2}\lg \kappa(z'), -S'\kappa(z') \rg \leq    -\frac{1}{2}\lg \kappa(z),S\kappa(z) \rg + C\eta |z-z'|^2.\]

Let us state  the comparison principle.
\begin{theorem}\label{comparisonthm}
Let  $w_1$ be a subsolution and $w_2$ be a supersolution of \eqref{eqpde2}, then $w_1\leq w_2$.
\end{theorem}

The rest of this subsection is devoted to the proof of this result. Let  $w_1$ be a subsolution and $w_2$ be a supersolution of \eqref{eqpde2} (recall that $w_1,w_2$ are bounded functions). Our aim is to show that $w_1\leq w_2$. We argue by contradiction, and assume that
\be\label{Contradiction}
M:=\sup_{z\in \Delta(K)\times\RR} \left\{ w_1(z)-w_2(z)\right\}\; >\; 0\;.
\ee
Because of the lack of compactness, let $\beta>0$ and $g(y):=\sqrt{(1+y^2)}$. Define 
\[M':=\sup_{z\in \Delta(K)\times\RR} \left\{ w_1(z)-w_2(z) - 2\beta g(y)\right\} .\]
We choose $\beta $ sufficiently small so that $M' >\frac{2C_1 \beta}{r}>0$ with $C_1= \|\kappa\|_\infty^2 + \|b \|_\infty$.
\p
We first regularize the maps $w_1$ and $w_2$ by quadratic sup and inf-convolution respectively. This technique is classical (see \cite{userguide} for details), for $\delta>0$ and $z\in \RR^{K+1}$ we define:
\[w_1^\delta(z):= \max_{z'\in \Delta(K)\times \RR} \left\{ w_1(z')-\frac{1}{2\delta}|z-z'|^2\right\}\]
and
\[w_{2,\delta}(z):= \min_{z'\in \Delta(K)\times\RR} \left\{ w_2(z')+\frac{1}{2\delta}|z-z'|^2\right\}.\]
Note that $w_1^\delta$ and $w_{2,\delta}$ are defined on the whole space $\RR^{K+1}$ and that $w_1^\delta$ is semiconvex while $w_{2,\delta}$ is semiconcave. 
Moreover, we have the following growth property (uniformly in $y$)
\begin{align*}\label{growthw1w2}
\lim_{|p|\to+\infty} |p|^{-1} w_1^\delta(p,y)=-\infty, \; \lim_{|p|\to+\infty}|p|^{-1} w_{2,\delta}(p,y)=+\infty\;.
\end{align*}
Define (with $z^i=(p^i,y^i)$):
\be\label{Mdelta}
M_{\delta}:= \sup_{z^1,z^2\in \RR^{K+1}}  \left\{ w_1^\delta(z^1)-w_{2,\delta}(z^2) - \beta (g(y^1)+g(y^2))  - \frac{1}{2 \delta}|z^1-z^2|^2\right\}.
\ee
The following result is classical.
\begin{lemma}\label{lem:limpdelta} For any $\delta>0$, the problem \eqref{Mdelta} has at least one maximum point. If $(z^1_\delta,z^2_\delta)$ is such a maximum point and if $(z^1)'_\delta\in  \Delta(K)\times \RR$ and $(z^2)''_\delta\in  \Delta(K)\times\RR$ are such that
\be\label{pprimepseconde}
w_1^\delta(z^1_\delta)= w_1((z^1)'_\delta)-\frac{1}{2\delta}|z^1_\delta-(z^1)'_\delta|^2\qquad {\rm and}\quad
 w_{2,\delta}(z^2_\delta)= w_2((z^2)''_\delta)+\frac{1}{2\delta}|z^2_\delta-(z^2)''_\delta|^2
\ee
then, as $\delta \to 0$, $M_\delta \to M'$ while
\[\frac{|z^1_\delta-z^2_\delta|^2}{2\delta}+\frac{|z^1_\delta-(z^1)'_\delta|^2}{2\delta}+ \frac{|z^2_\delta-(z^2)''_\delta|^2}{2\delta} \to 0.\]
\end{lemma}
We first prove that the regularized sub/supersolutions are sub/supersolutions of sligthly modified equations.
\begin{lemma}\label{lem:ineqw1w2} Assume that $w_1^\delta$ has a second order Taylor expansion at a point $z$. Then
\be \label{ineqw1}
\min \{ rw_1(z)+H(z',Dw_1^\delta(z), D^2w^1_{\delta}(z))  \; ; \;  - \lambda_{\max}( p',D^2_{p}w_1^\delta(z)) \} \leq 0,
\ee
where $z' \in \Delta(K)\times \RR$ is such that $ w_1^\delta(z)= w_1(z')-\frac{1}{2\delta}|z-z'|^2 $.\\
Similarly, if $w_{2,\delta}$ has a second order  Taylor expansion at a point $z$, then
\be 	\label{ineqw2}
rw_{2}(z)+H(z'',Dw_{2,\delta}(z), D^2w_{2,\delta}(z))  \geq 0,
\ee
where $z''\in \Delta(K)\times \RR$ is such that $ w_{2,\delta}(z)= w_2(z'')+\frac{1}{2\delta}|z-z''|^2$.
\end{lemma}

\begin{proof} We do the proof for $w_1^\delta$, the second part being similar.
Assume that  $w_1^\delta$ has a second order  Taylor expansion at a point $\bar z$ and set, for $\gamma>0$ small,
\[\phi_\gamma(z):= \lg  Dw_1^\delta(\bar z), z-\bar z\rg +\frac12\lg z-\bar z,D^2w_1^\delta(\bar z)(z-\bar z)\rg + \frac{\gamma}{2}|z-\bar z|^2.\]
Let $\bar z'$  denote a point in $\Delta(K)\times\RR$ such that $w_1^\delta(\bar z)= w_1(\bar z')-\frac{1}{2\delta}|\bar z-\bar z'|^2.$\\
Then $w_1^\delta-\phi_\gamma$ has a maximum at $\bar z$, which implies, by definition of $w_1^\delta$, that
\[w_1(z') -\frac{1}{2\delta}|z'-z|^2 \leq \phi_\gamma(z)-\phi_\gamma(\bar z)+ w_1^\delta(\bar z)\qquad \forall z\in \RR^{K+1}, \forall \ z'\in \Delta(K)\times \RR,\]
with an equality for $(z,z')=(\bar z,\bar z')$. If we choose $z=z'-\bar z'+\bar z$ in the above formula, we obtain:
\[w_1(z') \leq \phi_\gamma(z'-\bar z'+\bar z)+\frac{1}{2\delta}|\bar z'-\bar z|^2 -\phi_\gamma(\bar z)+ w_1^\delta(\bar z)\qquad \forall z'\in \Delta(K)\times\RR,\]
with an equality at $z'=\bar z'$. As $w_1$ is a subsolution, we obtain therefore, using the right-hand side of the above inequality as a test function,
\be \label{ineqw1'}
 \min\big\{   rw_1(\bar z')+H(\bar z',D\phi_\gamma(\bar z),D^2\phi_\gamma(\bar z)) \; ; \; - \lambda_{\max}(\bar{z}', D^2_{p}\phi_\gamma(\bar z))\big\} \leq 0.
\ee
By construction, we have $D\phi_\gamma(\bar z)= Dw_1^\delta(\bar z)$, $D^2\phi_\gamma(\bar z)= D^2w_1^\delta(\bar z)+\gamma I$ and $w_1(\bar z')\geq w_1^\delta(\bar z)$.
The conclusion follows therefore by letting $\gamma\to 0$.
\end{proof}

In order to use inequality \eqref{ineqw1}, we have to produce points at which $w_1^\delta$ is strictly concave with respect to the first variable.
For this reason, as in \cite{cardaetal}, we introduce a additional penalization. For $\sigma>0$ and $z^i=(p^i,y^i) \in \RR^{K+1}$, we consider
\[M_{\delta,\sigma} := \sup_{(z^1,z^2)\in (\RR^{K+1})^2}  \left\{ w_1^\delta(z^1)-w_{2,\delta}(z^2) - \beta (g(y^1)+g(y^2))+ \sigma g(|p^1|) - \frac{1}{2 \delta}|z^1-z^2|^2 \right\}.\]
One easily checks that there exists a maximizer $(\hat z^1,\hat{z}^2)$ to the above problem. In order to use Jensen's Lemma (Lemma A.3 in  \cite{userguide}), we also need  this maximum to be strict. For this we modify the penalization: we set for $i=1,2$:
\[ \xi_1(p^1)= g(|p^1|) - \sigma g(|p^1-\hat p^1|), \quad \zeta_i(y^i)=-\beta g(y^i) -\sigma g(y^i-\hat{y}^i) .\]
We choose $\sigma>0$ sufficiently small so that $\xi_1$ has a positive second order derivative. By definition,
\[M_{\delta,\sigma} =  \sup_{(z^1,z^2)\in (\RR^{K+1})^2}  \left\{  w_1^\delta(z^1)-w_{2,\delta}(z^2) +\zeta_1(y^1)+\zeta_2(y^2)+ \sigma \xi_1(|p^1|) - \frac{1}{2 \delta}|z^1-z^2|^2 \right\},\]
and the above problem has a strict maximum at $(\hat z^1, \hat{z}^2)$.
As the map $(z^1,z^2)\to w_1^\delta(z^1)-w_{2,\delta}(z^2) +\zeta_1(y^1)+\zeta_2(y^2)+ \sigma\xi_1(p^1) - \frac{1}{2 \delta}|z^1-z^2|^2$ is semiconcave, Jensen's Lemma (together with Alexandrov theorem) states that, for any $\ep>0$, there is vector $a_\ep\in (\RR^{K+1})^2$ with $|a_\ep|\leq \ep$, such that the problem
\[M_{\delta,\sigma,\ep}:= \sup_{z^1,z^2\in (\RR^{K+1})^2}  \left\{  w_1^\delta(z^1)-w_{2,\delta}(z^2) +\zeta_1(y^1)+\zeta_2(y^2)+ \sigma \xi_1(|p^1|) - \frac{1}{2 \delta}|z^1-z^2|^2+ \lg a_\ep,(z^1,z^2)\rg \right\},\]
has a maximum point $(z^1_{\delta,\sigma,\ep},z^2_{\delta,\sigma,\ep})\in (\RR^{K+1})^2$ at which the maps $w_1^\delta$ and $w_{2,\delta}$ have a second order  Taylor expansion.
From Lemma \ref{lem:ineqw1w2}, we have
\be\label{ineqw1delta}
 	\min\big\{   rw_1(z^1_{\delta,\sigma,\ep})+H((z^{1})'_{\delta,\sigma,\ep},Dw_1^\delta(z^1_{\delta,\sigma,\ep}),D^2w_1^\delta(z^1_{\delta,\sigma,\ep}))\  \,;\,  - \lambda_{\max}((z^{1})'_{\delta,\sigma,\ep}, D^2_{p}w_1^\delta(z^1_{\delta,\sigma,\ep}))\big\} \leq 0,
\ee
and
\be\label{ineqw2delta}
rw_{2}(z^2_{\delta,\sigma,\ep})+H((z^2)_{\delta,\sigma,\ep}'',Dw_{2,\delta}(z^2_{\delta,\sigma,\ep}),D^2w_{2,\delta}(z^2_{\delta,\sigma,\ep})) \geq 0,
\ee
where $(z^1)'_{\delta,\sigma,\ep}$ and $(z^2)''_{\delta,\sigma,\ep}$ are points in $\Delta(K)\times \RR$ at which one has
\[ w_1^\delta(z^1_{\delta,\sigma,\ep})= w_1((z^1)'_{\delta,\sigma,\ep})-\frac{1}{2\delta}|z^1_{\delta,\sigma,\ep}-(z^1)_{\delta,\sigma,\ep}'|^2\; {\rm and }\;
 w_{2,\delta}(z^2_{\delta,\sigma,\ep})= w_2((z^2)''_{\delta,\sigma,\ep})+\frac{1}{2\delta}|z^2_{\delta,\sigma,\ep}-(z^2)''_{\delta,\sigma,\ep}|^2 .\]
Using the properties of inf and sup-convolutions, we have:
\be\label{Dw1=Dw2=}
Dw_1^\delta(z^1_{\delta,\sigma,\ep})=-\frac{1}{\delta}\left(z^1_{\delta,\sigma,\ep}-(z^1)_{\delta,\sigma,\ep}'\right)\;{\rm and} \;
D w_{2,\delta}(z^2_{\delta,\sigma,\ep})= \frac{1}{\delta}\left(z^2_{\delta,\sigma,\ep}-(z^2)_{\delta,\sigma,\ep}''\right).
\ee
By definition of $M_{\delta,\sigma,\ep}$ we have for all $(z^1,z^2) \in (\RR^{K+1})^2$:
\[w_1^\delta(z^1)- w_{2,\delta}(z^2)+\zeta_1(y^1)+\zeta_2(y^2)+ \sigma \xi_1(p_1)  \leq M_{\delta,\sigma,\ep}+ \frac{1}{2 \delta}|z^1-z^2|^2 -  \lg a_\ep,(z^1,z^2)\rg,\]
with an equality at $(z^1_{\delta,\sigma,\ep},z^2_{\delta,\sigma,\ep})$. Hence

\be\label{eqDw1Dw2}
Dw_1^\delta(z^1_{\delta,\sigma,\ep})+
  \left(\begin{array}{c} \sigma D\xi_1(p^1_{\delta,\sigma,\ep})\\ -\beta g' (y^1_{\delta,\sigma,\ep})- \sigma g' (y^1_{\delta,\sigma,\ep} - \hat{y}^1)\end{array} \right) = \frac{1}{\delta}(z^1_{\delta,\sigma,\ep}-z^2_{\delta,\sigma,\ep}) - a^1_\ep
\ee
\be\label{eqDw1Dw22}
-Dw_2^\delta(z^2_{\delta,\sigma,\ep})+
 \left(\begin{array}{c} 0\\ -\beta g' (y^2_{\delta,\sigma,\ep})- \sigma g' (y^2_{\delta,\sigma,\ep} - \hat{y}^1)\end{array} \right) = \frac{1}{\delta}(z^2_{\delta,\sigma,\ep}-z^1_{\delta,\sigma,\ep}) - a^2_\ep
\ee
while 
\be\label{ineqD2w1vsD2w2}
 \begin{pmatrix} S & 0 \\ 0 & S' \end{pmatrix}\leq \frac{1}{\delta} \begin{pmatrix}I & -I \\ -I & I \end{pmatrix} 
\ee
with 
\[ S:=D^2w_1^\delta(z^1_{\delta,\sigma,\ep}) + \left(\begin{array}{cc}  \sigma D^2 \xi_1(p^1_{\delta,\sigma,\ep})&0\\0&  -\beta g'' (y^1_{\delta,\sigma,\ep})- \sigma g'' (y_{\delta,\sigma,\ep} - \hat{y}^1)\end{array} \right)\]
\[ S':=-D^2w_{2,\delta}(z^2_{\delta,\sigma,\ep})+ \left(\begin{array}{cc} 0&0\\0&  -\beta g'' (y^2_{\delta,\sigma,\ep})- \sigma g'' (y^2_{\delta,\sigma,\ep} - \hat{y}^2)\end{array}\right)\]
This implies that $S \leq -S'$ (see \cite{userguide} p.19) and therefore 
\be \label{ineqconcave}
 D^2_p w_1^\delta(z^1_{\delta,\sigma,\ep}) \leq D^2_p w_{2,\delta}(z^2_{\delta,\sigma,\ep})- \sigma D^2 \xi_1(p^1_{\delta,\sigma,\ep}).
\ee
We now check that $\ds \lambda_{\max}(((p^1)_{\delta,\sigma,\ep}'),D^2_{p}w_1^\delta(z^1_{\delta,\sigma,\ep}))<0$. Using the definition of $w_{2,\delta}$, for all $p\in \RR^{K}$ and $p''\in \Delta(K)$,
\[w_{2,\delta}(p,y^2_{\delta,\sigma,\ep})\leq w_2(p'',(y^2)_{\delta,\sigma,\ep}'')+\frac{1}{2\delta}\left(|p-p''|^2+
|y^2_{\delta,\sigma,\ep}-(y^2)_{\delta,\sigma,\ep}''|^2\right),\]
with an equality at $(p, p'')= (p^2_{\delta,\sigma,\ep}, (p^2)_{\delta,\sigma,\ep}'')$. If $m\in T_{\Delta(K)}((p^2)_{\delta,\sigma,\ep}'')$ with $|m|$ small enough, taking $p:= p^2_{\delta,\sigma,\ep}+m$ and $p''= (p^2)_{\delta,\sigma,\ep}''+ m$ gives
\be
w_{2,\delta}(p^2_{\delta,\sigma,\ep}+m,y^2_{\delta,\sigma,\ep})\leq  w_2((p^2)_{\delta,\sigma,\ep}''+ m,(y^2)_{\delta,\sigma,\ep}'')+\frac{1}{2\delta}\left(|p^2_{\delta,\sigma,\ep}-(p^2)_{\delta,\sigma,\ep}''|^2+
|y^2_{\delta,\sigma,\ep}-(y^2)_{\delta,\sigma,\ep}''|^2\right),
\ee
with equality for $m=0$. As $w_2$ is concave with respect to the first variable (see e.g. Lemma 3.2 in \cite{cardadouble}), the above inequality implies that  $\ds \lambda_{\max}((p^2)_{\delta,\sigma,\ep}'',D^2_{p}w_{2,\delta}(z^2_{\delta,\sigma,\ep}))\leq 0$. In view of \eqref{ineqconcave} we get therefore
\[\ds \lambda_{\max}((p^1)_{\delta,\sigma,\ep}',D^2_{p}w_1^\delta(z^1_{\delta,\sigma,\ep})) \leq -\sigma \lambda_{\min}((p^1)_{\delta,\sigma,\ep}',D^2\xi_1(p^1_{\delta,\sigma,\ep}))<0,\]
because $D^2\xi_1>0$ by construction.
So \eqref{ineqw1delta}  becomes
\be\label{ineqw1deltaNew}
 rw_1(z^1_{\delta,\sigma,\ep})+H((z^1)'_{\delta,\sigma,\ep}),Dw_1^\delta(z^1_{\delta,\sigma,\ep}),D^2w_1^\delta(z^1_{\delta,\sigma,\ep})) \leq 0
\ee
We compute the difference of the two inequalities \eqref{ineqw1deltaNew} and \eqref{ineqw2delta}  above:
\begin{align*}
r(w_1^\delta(z^1_{\delta,\sigma,\ep})&-w_{2,\delta}(z^2_{\delta,\sigma,\ep})) 
+H((z^1)_{\delta,\sigma,\ep}',Dw_1^\delta(z^1_{\delta,\sigma,\ep}),D^2w_1^\delta(z^1_{\delta,\sigma,\ep})) \\ &-H((z^2)_{\delta,\sigma,\ep}'',Dw_{2,\delta}(z^2_{\delta,\sigma,\ep}),D^2w_{2,\delta}(z^2_{\delta,\sigma,\ep})) \leq 0,
\end{align*}
where, in view of \eqref{Dw1=Dw2=} and the definitions of $S,S'$ (and using that $g',g''$ and $|D\xi_1|$, $|D^2\xi_1|$ are bounded by $1$)
\begin{align*}
H((z^1)_{\delta,\sigma,\ep}',Dw_1^\delta(z^1_{\delta,\sigma,\ep}),D^2w_1^\delta(z^1_{\delta,\sigma,\ep})) \geq
H((z^1)_{\delta,\sigma,\ep}',\frac{1}{\delta}(z^1_{\delta,\sigma,\ep}-z^2_{\delta,\sigma,\ep}),S)  - C_1 (\beta+\varepsilon +\sigma), 
\end{align*}
\begin{align*}
H((z^2)_{\delta,\sigma,\ep}'',Dw_{2,\delta}(z^2_{\delta,\sigma,\ep}),D^2w_{2,\delta}(z^2_{\delta,\sigma,\ep}))\leq
H((z^2)''_{\delta,\sigma,\ep},\frac{1}{\delta}(z^1_{\delta,\sigma,\ep}-z^2_{\delta,\sigma,\ep}),-S') +C_1 (\beta + \varepsilon +\sigma), 
\end{align*}
Next, we have:
\begin{align*}
|u((z^1)_{\delta,\sigma,\ep}')-u((z^2)_{\delta,\sigma,\ep}'')| \leq C |(z^1)_{\delta,\sigma,\ep}'-(z^2)_{\delta,\sigma,\ep}''|,
\end{align*}
\begin{align*}
| \lg b((z^1)_{\delta,\sigma,\ep}'),\frac{1}{\delta}(z^1_{\delta,\sigma,\ep}-z^2_{\delta,\sigma,\ep}) \rg  - \lg b(z^2)_{\delta,\sigma,\ep}''), \frac{1}{\delta}(z^1_{\delta,\sigma,\ep}-z^2_{\delta,\sigma,\ep}) \rg  | \leq \frac{C}{\delta} |(z^1)_{\delta,\sigma,\ep}'-(z^2)_{\delta,\sigma,\ep}''||z^1_{\delta,\sigma,\ep}-z^2_{\delta,\sigma,\ep}|,
\end{align*}
\begin{align*}
 -\frac{1}{2}\lg \kappa((z^2)''_{\delta,\sigma,\ep}), -S'\kappa((z^2)''_{\delta,\sigma,\ep}) \rg \leq    -\frac{1}{2}\lg \kappa((z^1)_{\delta,\sigma,\ep}'),S\kappa((z^1)_{\delta,\sigma,\ep}') \rg + \frac{C}{\delta} |(z^1)_{\delta,\sigma,\ep}'-(z^2)_{\delta,\sigma,\ep}''|^2.
\end{align*}

We deduce that:
\begin{align*}
r(w_1^\delta(p_{\delta,\sigma,\ep})-w_{2,\delta}(p_{\delta,\sigma,\ep})) &\leq
 C\left(\frac{1}{\delta}\left|(z^1)_{\delta,\sigma,\ep}'-(z^2)_{\delta,\sigma,\ep}''\right|^2+\left|(z^1)_{\delta,\sigma,\ep}'-(z^2)_{\delta,\sigma,\ep}''\right| \right)\\
 & +\frac{C}{\delta} |(z^1)_{\delta,\sigma,\ep}'-(z^2)_{\delta,\sigma,\ep}''||z^1_{\delta,\sigma,\ep}-z^2_{\delta,\sigma,\ep}| +2C_1( \beta+\varepsilon +\sigma).
\end{align*}
As $\sigma$ and $\ep$ tend to $0$, the $z^1_{\delta,\sigma,\ep}$, $z^2_{\delta,\sigma,\ep}$,  $(z^1)_{\delta,\sigma,\ep}'$ and $(z^2)_{\delta,\sigma,\ep}''$ converges (up to a subsequence) to  $z^1_\delta$, $z^2_\delta$, $(z^1)_\delta'$ and $(z^2)_\delta''$, where $(z^1_\delta,z^2_{\delta})$ is a maximum in \eqref{Mdelta} and where $(z^1)_\delta'$ and $(z^2)_\delta''$ satisfy \eqref{pprimepseconde}. The above inequality together with the definition of $M_\delta$ implies:
\begin{align*}
rM_\delta \leq r(w_1^\delta(z^1_\delta)-w_{2,\delta}(z^2_\delta))\leq &
C\left(\frac{1}{\delta}\left|(z^1)_{\delta}'-(z^2)_{\delta}''\right|^2+\left|(z^1)_{\delta}'-(z^2)_{\delta}''\right| \right)\\
 & +\frac{C}{\delta} |(z^1)_{\delta}'-(z^2)_{\delta}''||z^1_{\delta}-z^2_{\delta}| +2C_1 \beta.
 \end{align*}
We finally let $\delta\to 0$: in view of Lemma \ref{lem:limpdelta} the above inequality yields to $rM'=\lim_{\delta\to 0}rM_\delta\leq 2C_1 \beta$, which contradicts our initial assumption. Therefore $w_1\leq w_2$ and the proof is complete.

\end{document}